%% file: Hitchin70.tex
 \documentclass[twosided,a4paper]{amsart}

\input{preamble}

\newcommand{\smtrx}[1]{\left (\begin{smallmatrix}#1\end{smallmatrix}\right)}
 \usepackage[margin=1.2 in]{geometry}
\author{Brian Collier}
\address{Department of Mathematics \\
University of Maryland \\
4176 Campus Drive \\
College Park, MD 20742}
\email{bcollie2@math.umd.edu}

\title[Various generalizations of $\sP\sSL(2,\R)$-Higgs bundles]{Various generalizations and deformations of $\sP\sSL(2,\R)$ surface group representations and their Higgs bundles}
\date{\today}
\begin{document}

\setlength{\smallskipamount}{6pt}
\setlength{\medskipamount}{10pt}
\setlength{\bigskipamount}{16pt}
\begin{abstract}
Recall that the group $\sP\sSL(2,\R)$ is isomorphic to $\sP\sSp(2,\R),\ \sSO_0(1,2)$ and $\sP\sU(1,1).$ The goal of this paper is to examine the various ways in which Fuchsian representations of the fundamental group of a closed surface of genus $g$ into $\sP\sSL(2,\R)$ and their associated Higgs bundles generalize to the higher rank groups $\sP\sSL(n,\R),\ \sP\sSp(2n,\R),\ \sSO_0(2,n),\ \sSO_0(n,n+1)$ and $\sP\sU(n,n)$. For the $\sSO_0(n,n+1)$-character variety, we parameterize $n(2g-2)$ new connected components as the total space of vector bundles over appropriate symmetric powers of the surface and study how these components deform in the $\sSO_0(n,n+2)$-character variety. This generalizes results of Hitchin for $\sP\sSL(2,\R)$.
\end{abstract}
\maketitle
\section{Introduction}

Since Hitchin introduced Higgs bundles, they have been effectively used to count connected components of the character variety of surface group representations in a reductive Lie group. 
Even more, when one is lucky, Higgs bundles can be used to explicitly parameterize certain connected components of the character variety.  
For a closed surface $S$ with genus $g\geq2$, Hitchin \cite{selfduality} gave an explicit parameterization of all but one of the connected components of the character variety of conjugacy classes of representations of the fundamental group of $S$ in the Lie group $\sP\sSL(2,\R).$ 
Namely, he showed that each component with {\em nonzero} Euler class is diffeomorphic to the total space of a smooth vector bundle over an appropriate symmetric product of the surface. When the Euler class obtains its maximal value, this recovers the classical parameterization of the set of {\em Fuchsian representations} (which is identified with the Teichm\"uller space of $S$) as a vector space of complex dimension $3g-3.$  
The component with zero Euler class is the only component which contains representations with compact Zariski closure.

Hitchin later showed that the $\sP\sSL(n, \R)$-character variety has three connected components\footnote{More precisely, the character variety of $\sP\sSL(2n,\R)$ has six components which come in three isomorphic pairs. These pairs are identified via the action of the outer automorphism group of $\sP\sSL(2n,\R).$}, two in which all representations can be deformed to compact representations and one in which {\em no} representation can be deformed to a compact representation \cite{liegroupsteichmuller}. Moreover, this last component, now called the Hitchin component, can be interpreted as a deformation space of Fuchsian representations, and thus, generalizes the Teichm\"uller space of $S$ to the $\sP\sSL(n,\R)$-character variety.
In fact, Hitchin parameterized the Hitchin component by a vector space of holomorphic differentials on the surface $S$ equipped with a Riemann surface structure. 

The focus of this paper is to describe other generalizations of surface group representations into $\sP\sSL(2,\R)$ and their associated Higgs bundles by using the low dimensional isomorphisms:
\[\sP\sSL(2,\R)\cong\sP\sSp(2,\R)\cong\sSO_0(2,1)\cong\sP\sU(1,1)~.\]
Namely, we will consider groups locally isomorphic to
\[\xymatrix{\sP\sSL(n,\R)~,&\sP\sSp(2n,\R)~,&\sSO_0(2,n)~,&\sSO_0(n,n+1)&\text{and}&\sP\sU(n,n)}~.\]
All of these families of groups are {\em split} real groups ($\sP\sSL(n,\R),~ \sP\sSp(2n,\R),~ \sSO_0(n,n+1)$) or groups of {\em Hermitian type} $(\sP\sSp(2n,\R),~\sSO_0(2,n),~\sP\sU(n,n))$. 
Hitchin's description of the $\sP\sSL(n,\R)$-Hitchin component can be adapted to any split real group. 
Thus, for each of the split groups, there is a special connected component, also called the Hitchin component, which can be thought of as a deformation space of Fuchsian representations. 
In the Hermitian case, there is an integer invariant called the Toledo invariant which generalizes the Euler class. This invariant again is bounded in absolute value, and the representations with {\em maximal Toledo invariant} are of particular interest.  
Indeed, Hitchin representations and maximal representations are the only known connected components of surface group character varieties which consist entirely of representations that satisfy Labourie's Anosov condition \cite{AnosovFlowsLabourie,MaxRepsAnosov}.

For the group $\sP\sU(n,n),$ the space of maximal representations in connected \cite{MarkmanXia}. Moreover, all maximal representations lift to $\sSU(n,n)$-representations, and the space of maximal $\sSp(2n,\R)$-representations is a subset of the space of maximal $\sSU(n,n)$-representations. We will not focus on maximal $\sP\sU(n,n)$-representations, but will discuss the case of maximal $\sSp(2n,\R)$-representations.

Maximal representations into $\sSp(2n,\R)$ have been studied by many authors from various perspectives. For $n\geq3,$ the space maximal $\sSp(2n,\R)$-representations all behave in a similar manner. Namely, there are $3\cdot 2^{2g}$ connected components of maximal representations \cite{HiggsbundlesSP2nR} and every component can be interpreted as a deformation space of Fuchsian representations \cite{TopInvariantsAnosov} (see section \ref{section Hitchin maximal}). 
However, when $n=2$ the space of maximal $\sSp(4,\R)$-representations behave quite differently than the general case. In particular, there are many more connected components \cite{sp4GothenConnComp}, and there are connected components which are smooth and consist entirely of Zariski dense representations \cite{MaximalSP4}. In other words, there are connected components of maximal $\sSp(4,\R)$-representations which cannot be interpreted as deformation spaces of Fuchsian representations. 

We will show how the strange behavior of maximal $\sSp(4,\R)$-representations can be interpreted as a consequence of the low dimensional isomorphism $\sSO_0(2,3)\cong\sP\sSp(4,\R)$. Namely, for each integer $d\in(0,n(2g-2)]$, we will construct a connected component of the $\sSO_0(n,n+1)$-character variety which directly generalizes the exceptional components of maximal $\sSp(4,\R)$-representations (see Theorem \ref{THM SOnn+1comp}). Each of these components is parameterized by a smooth vector bundle over an appropriate symmetric product of the surface. Moreover, when the integer $d$ is maximal (i.e. $d=n(2g-2)$), this component recovers the $\sSO_0(n,n+1)$-Hitchin component. 

Note that when $n\geq3$, $\sSO_0(n,n+1)$ is {\em not} a group of Hermitian type, so these new connected components do not arise from the maximality of a known topological invariant. 
Interestingly, the special features of maximal $\sSO_0(2,3)$ representations can be viewed as a generalization of Hitchin's description of the $\sP\sSL(2,\R)$-character variety by using the isomorphism $\sP\sGL(2,\R)=\sSO(2,1)$ (see section \ref{Section PGL2}). 
Finally, we will study how these exotic connected components of the $\sSO_0(n,n+1)$-character variety deform in the $\sSO_0(n,n+2)$-character variety.
 

 Sections 2, 3, and 4 are mostly a survey of known results. Sections 5, and 6 contain new results, some of which were part of the author's thesis, and will appear in \cite{SO23LabourieConj}, \cite{SOnn+1PosRepsHiggs}, and \cite{SOpqStabilityAndMinima}.
\smallskip

 \textbf{Acknowledgments} The author's research is supported by the National Science Foundation under Award No. 1604263.

\section{Surface group representations and Higgs bundles}\label{section background}
Let $S$ be an orientable closed surface of genus $g\geq 2$ and $\sG$ be a connected real {\em (semi)simple} algebraic Lie group. 
Denote the fundamental group of $S$ by $\Gamma.$
The set of representations of $\Gamma$ in $\sG$ is defined to be the set of group homomorphisms $\Hom(\Gamma,\sG).$ 
Since $\sG$ is algebraic and $\Gamma$ has a finite presentation, $\Hom(\Gamma,\sG)$ can be given the structure of an algebraic variety.
A representation $\rho\in\Hom(\Gamma,\sG)$ is called {\em reductive} if the Zariski closure of $\rho(\Gamma)$ is a reductive subgroup of $\sG.$ Denote the space of reductive representations by $\Hom^+(\Gamma,\sG).$
\begin{Definition}
	The $\sG$-character variety $\Xx(\Gamma,\sG)$ is the space $\Xx(\Gamma,\sG)=\Hom^+(\Gamma,\sG)/\Inn(\sG)$
	where $\Inn(\sG)$ denotes the set of inner automorphism of $\sG.$ 
	\end{Definition}
The $\sG$-character variety is an algebraic variety of dimension $\mathrm{dim}(\sG)(2g-2)$ (see \cite{SymplecticNatureofFund}).

\begin{Example}\label{EX: Fuchsian reps}
The set of {\em Fuchsian representations} $\Fuch(\Gamma)\subset\Xx(\Gamma,\sP\sSL(2,\R))$ is defined to be the subset of conjugacy classes of {\em faithful} representations with {\em discrete image}.  
The space $\Fuch(\Gamma)$ defines two isomorphic connected components of $\Xx(\Gamma,\sP\sSL(2,\R))$ \cite{TopologicalComponents} and is in one to one correspondence with the Teichm\"uller space of isotopy classes of marked Riemann surface structures on the orientable surface $S.$

\end{Example}
Recall that the universal cover $\widetilde S\ra S$ is a principal $\Gamma$-bundle. Thus, associated to a representation $\rho:\Gamma\ra\sG,$ there is a flat principal $\sG$ bundle $\widetilde S\times_\rho\sG$ on $S.$ 
For connected groups, topological $\sG$-bundles on $S$ are classified by a characteristic class 
\[\omega\in H^2(S,\pi_1\sG)\cong\pi_1\sH\] where $\sH\subset \sG$ is a maximal compact subgroup. 
Thus, the $\sG$-character variety decomposes as 
\[\Xx(\Gamma,\sG)=\bigsqcup\limits_{\omega\in\pi_1\sH}\Xx^\omega(\Gamma,\sG)\]
where the equivalence class of a reductive representation $\rho:\Gamma\ra\sG$ lies in $\Xx^\omega(\Gamma,\sG)$ if and only if the flat $\sG$ bundle determined by $\rho$ has topological type determined\footnote{For nonconnected groups, this is not true. In the nonconnected case, there are two characteristic classes, one in $H^1(S,\pi_0(\sH))$ and one in $H^2(S,\pi_1\sH))$. As we will see for $\sG=\sSO(2,1)$, the topological type of a $\sG$-bundle is not always uniquely determined by these invariants. See \cite{AndrePGLnR} and \cite{Oliveira_GarciaPrada_2016} for more details on nonconnected groups.} by $\omega\in\pi_1\sH.$

In general, the number of connected components of the character variety of a simple Lie group $\sG$ has not been established. However, there have been many partial results. For instance, when $\sG$ is compact and semisimple, the spaces $\Xx^\omega(\Gamma,\sG)$ are connected and nonempty \cite{ramanathan_1975}. This implies the following proposition.
\begin{Proposition}
	If $\sG$ is a connected real semisimple Lie group such that the maximal compact subgroup $\sH$ is semisimple, then, for each $\omega\in\pi_1\sH,$ the space $\Xx^{\omega}(\sG)$ is nonempty. Moreover, each component $\Xx^\omega(\Gamma,\sG)$ contains a unique connected component with the property that every representation in it can be continuously deformed to a representation with compact Zariski closure. 
\end{Proposition}
The above proposition implies that, when $\sG$ is a connected semisimple {\em complex} Lie group, the space $\Xx^\omega(\Gamma,\sG)$ is nonempty for each $\omega\in\pi_1\sH$. In fact, Li proved that for complex semisimple Lie groups, each of the spaces $\Xx^\omega(\Gamma,\sG)$ is connected \cite{JunLiConnectedness}. In particular, we have the following:
\begin{Corollary}
     	If $\sG$ is a semisimple complex Lie group, then any representation $\rho\in\Xx(\Gamma,\sG)$ can be continuously deformed to a representation with compact Zariski closure. 
     \end{Corollary}     

A semisimple Lie group $\sG$ whose maximal compact subgroup is not semisimple but only reductive is called {\em a group of Hermitian type.} 
When $\sG$ is simple and of Hermitian type, the center of the maximal compact subgroup has dimension one and defines a subgroup of $\pi_1\sH$ which is isomorphic to $\Z.$ 
For example, $\sSp(2n,\R)$ is a group of Hermitian type since the maximal compact subgroup of $\sSp(2n,\R)$ is $\sU(n)$ and $\pi_1\sU(n)\cong\Z.$ In the Hermitian case, the character variety decomposes as 
 	\[\Xx(\Gamma,\sG)=\bigsqcup\limits_{\tau\in\Z}\Xx^\tau(\sG).\]
Moreover, the spaces $\Xx^\tau(\sG)$ and $\Xx^{-\tau}(\sG)$ are isomorphic and nonempty for only {\em finitely} many values of $\tau\in\Z$. 
Let $M$ be the largest value of $\tau$ such that $\Xx^\tau(\sG)$ is nonempty. The set of representations in $\Xx^M(\sG)$ is called the set of {\em maximal representations}. The value of $M$ depends only on the real rank of the group $\sG$, the topology of $S$ and a choice of normalization.
\begin{Example}
	For $\sG=\sP\sSL(2,\R),$ we have $\pi_1\sH=\pi_1\sSO(2)=\Z.$ Thus, 
	\[\Xx(\Gamma,\sP\sSL(2,\R))=\bigsqcup\limits_{\tau\in\Z}\Xx^\tau(\Gamma,\sP\sSL(2,\R)).\]
	In this case, the invariant $\tau$ is the Euler class of the circle bundle associated to a flat $\sP\sSL(2,\R)$-bundle. This Euler class satisfies the Milnor-Wood inequality \cite{MilnorMWinequality}: 
	\[|\tau|\leq2g-2.\]
	Moreover, the maximal components $\Xx^{2g-2}(\Gamma,\sP\sSL(2,\R))\sqcup\Xx^{-2g+2}(\Gamma,\sP\sSL(2,\R))$ correspond to the set of Fuchsian representations from Example \ref{EX: Fuchsian reps} \cite{TopologicalComponents}. 

 \end{Example}
\subsection{$\sG$-Higgs bundles}
Unlike the character variety, to describe Higgs bundles we must fix a Riemann surface structure $X$ on the topological surface $S.$
As before, let $\sG$ be a real algebraic simple Lie group with Lie algebra $\fg,$ and fix $\sH\subset\sG$ a maximal compact subgroup with Lie algebra $\fh$. 
Let $\fg=\fh\oplus\fm$ be the corresponding Cartan decomposition of the Lie algebra $\fg.$ 
The splitting $\fg=\fh\oplus\fm$ is invariant with respect to the adjoint action of $\sH$ on $\fg.$
Complexifying everything gives an $Ad_{\sH_\C}$ invariant decomposition $\fg_\C=\fh_\C\oplus\fm_\C.$
If $P$ is a principal $\sG$-bundle and $\alpha:\sG\ra\sGL(V)$ is a linear representation, denote the associated vector bundle $P\times_{\sG} V$ by $P[V].$

\begin{Definition}
	A {\em $\sG$-Higgs bundle} is a pair $(\Pp,\varphi)$ where $\Pp$ is a holomorphic principal $\sH_\C$-bundle and $\varphi\in H^0(X,\Pp[\fm_\C]\otimes K)$ is a holomorphic section of the associated $\fm_\C$-bundle twisted by the canonical bundle $K$ of $X.$
\end{Definition}

\begin{Example}
If $\sG$ is compact, then $\sH_\C=\sG_\C$ and $\fm_\C=\{0\}.$ Thus for compact groups, a Higgs bundle is the same as a holomorphic principal $\sG_\C$-bundle. 
\end{Example}

If $\alpha:\sH_\C\ra\sGL(V)$ is a linear representation of $\sH_\C$, then the data of a $\sG$-Higgs bundle can be described by the vector bundle associated to $\alpha$ and a section of another associated bundle. For instance, when $\sG=\sSL(n,\C)$ and $\alpha:\sSL(n,\C)\ra\sGL(\C^n)$ is the standard representation, an $\sSL(n,\C)$-Higgs bundle $(\Pp,\varphi)$ defines the data of a rank $n$ holomorphic vector bundle $E\ra X$ with $\Lambda^n(E)=\Oo$ and a {\em traceless} holomorphic section $\Phi$ of $\End(E)\otimes K$.
This allows us to define the notion of stability and polystability of an $\sSL(n,\C)$-Higgs bundle.
\begin{Definition}
	An $\sSL(n,\C)$-Higgs bundle $(E,\Phi)$ is {\em stable} if, for all subbundles $F\subset E$ with $\Phi(F)\subset F\otimes K$, we have $deg(F)<0;$ $(E,\Phi)$ is {\em polystable} if it is a direct sum of stable $\sGL(n_j,\C)$-Higgs bundles of degree zero. 
 \end{Definition}

There are appropriate notions of (semi)stability and polystability for $\sG$-Higgs bundle with which the moduli space of $\sG$-Higgs bundles can be defined as a polystable quotient. 
Rather than recalling the definition of polystability for $\sG$-Higgs bundles, we will use the following result (see \cite{HiggsPairsSTABILITY}). 

\begin{Proposition}
	Let $\sG$ be a real form of a complex subgroup of $\sSL(n,\C)$. A $\sG$-Higgs bundle $(\Pp,\varphi)$ is polystable if and only if the associated $\sSL(n,\C)$-Higgs bundle is polystable. 
\end{Proposition} 

Recall that holomorphic structures on a smooth principal $\sH_\C$-bundle $P\ra X$\ are equivalent to Dolbeault operators $\bar\p_P\in \Omega^{0,1}(X,P[\fh_\C])$.  
The gauge group $\Gg_{\sH_\C}(P)$ of smooth bundle automorphisms of $P$ acts on the set of Higgs bundle structures $(\Pp,\varphi)=(\bar\p_P,\varphi)$ by the adjoint action.

\begin{Definition} 
	The {\em moduli space of $\sG$-Higgs bundles} $\Mm(\sG)$ on $X$ is defined as the set of isomorphism class of polystable $\sG$-Higgs bundles.
\end{Definition}
In fact, the space $\Mm(\sG)$ can be given the structure of a complex analytic variety of expected dimension $dim(\sG)(2g-2)$ \cite{selfduality,localsystems,schmitt_2005}. As with the character variety, for connected groups, the topological type of the $\sH_\C$ bundle of a Higgs bundle $(\Pp,\varphi)$ is determined by a class $\omega\in\pi_1\sH.$ If $\Mm^\omega(\sG)$ denotes the set of $\sG$-Higgs bundles with topological invariant $\omega\in\pi_1\sH$, then the moduli space $\Mm(\sG)$ decomposes as 
\[\Mm(\sG)=\bigsqcup\limits_{\omega\in\pi_1\sH}\Mm^{\omega}(\sG)~.\]
Moreover, we have the following fundamental result which we will use to go back and forth between statements about the Higgs bundle moduli space and the character variety.
\begin{Theorem}\label{THM: NAHC}
	Let $X$ be a Riemann surface with genus at least two and fundamental group $\Gamma.$ Let $\sG$ be a real simple Lie group with maximal compact subgroup $\sH.$ The moduli space $\Mm(\sG)$ of $\sG$-Higgs bundles on $X$ is homeomorphic to the $\sG$-character variety $\Xx(\Gamma,\sG)$. Moreover, for each $\omega\in\pi_1\sH$, the components $\Mm^\omega(\sG)$ and $\Xx^\omega(\Gamma,\sG)$ are homeomorphic.
\end{Theorem}
\begin{Remark}
When $\sG$ is compact, Theorem \ref{THM: NAHC} was proven using the theory of stable holomorphic bundles by Narasimhan and Seshedri \cite{NarasimhanSeshadri} for $\sG=\sSU(n),$ and Ramanathan \cite{ramanathan_1975} in general. For $\sG$ noncompact, it was proven Hitchin \cite{selfduality} and Donaldson \cite{harmoicmetric} for $\sG=\sSL(2,\C)$ and Simpson \cite{SimpsonVHS} and Corlette \cite{canonicalmetrics} in general using the theory of Higgs bundles. Theorem \ref{THM: NAHC} holds more generally for real reductive Lie groups \cite{HiggsPairsSTABILITY}.
We will use this correspondence to, amongst other things, study the topology of the character variety.
\end{Remark}

\subsection{Vector bundle description for $\sSL(n,\R)$, $\sSp(2n,\R)$ and $\sSO(p,q)$-Higgs bundle} We now give vector bundle definitions for certain $\sG$-Higgs bundles and describe the topological invariants.

\smallskip 
\noindent\textbf{$\sSL(n,\R)$-Higgs bundles:\ }
The maximal compact subgroup of $\sSL(n,\R)$ is isomorphic to $\sSO(n)$ and the Lie algebra $\fsl(n,\R)$ consists of traceless $(n\times n)$-matrices. Let $Q$ be a positive definite symmetric quadratic form on $\R^n$. The Cartan decomposition of $\fsl(n,\R)$ is
$\fsl(n,\R)=\fso(n)\oplus sym_0(\R^n)$
where 
\[\fso(n)=\{X\in\fsl(n,\R)\ |\ X^TQ+QX=0\}\ \ \ \ \ \ \ \text{and}\ \ \ \ \ \ \\sym_0(\R^n)=\{X\in\fsl(n,\R)\ |\ X^TQ=QX\}~.\] 
Complexifying everything gives $\fsl(n,\C)=\fso(n,\C)\oplus sym_0(\C^n).$
\begin{Definition}
  	An {\em $\sSL(n,\R)$-Higgs bundle} is pair $(E,\Phi)$ where $E$ is a rank $n$ holomorphic bundle with an orthogonal structure $Q$ such that $\Lambda^nE=\Oo$ and $\Phi\in H^0(X,\End(E)\otimes K)$ is traceless and symmetric with respect to $Q,$ i.e. $\Phi^TQ=Q\Phi.$
  \end{Definition}  
For $n>2$, $\pi_1\sSO(n)=\Z_2$ and the moduli space of $\sSL(n,\R)$-Higgs bundles decomposes as
\[\Mm(\sSL(n,\R))=\bigsqcup\limits_{\omega\in \Z_2}\Mm^\omega(\sSL(n,\R)).\]
The invariant $\omega\in\Z_2$ of an $\sSL(n,\R)$-Higgs bundle $(E,\Phi)$ is the second Steifel-Whitney class of the orthogonal bundle $E.$
\smallskip

\noindent\textbf{$\sSp(2n,\R)$-Higgs bundles:\ }Consider the symplectic form $\Omega=\smtrx{0&Id\\-Id&0}$ on $\C^{2n}.$ 
The complex symplectic group $\sSp(2n,\C)$ consists of linear transformations $g\in\sGL(2n,\C)$ such that $g^T\Omega g=\Omega.$ 
The Lie algebra $\fsp(2n,\C)$ consists of matrices $X$ such that $X^T\Omega+\Omega X=0.$ Such an $X\in\fsp(2n,\C)$ is given by $X=\smtrx{A&B\\C&-A^T}$ where $A,$ $B$ and $C$ are $n\times n$ complex matrices with $B$ and $C$ symmetric.

One way of defining the group $\sSp(2n,\R)$ is as the subgroup of $\sSp(2n,\C)$ consisting of matrices with real entries. 
However, when dealing with $\sSp(2n,\R)$-Higgs bundles it will be useful to consider $\sSp(2n,\R)$ as the fixed point set of a conjugation $\lambda$ which acts by 
$\lambda(g)=\smtrx{0&I\\I&0}\overline g\smtrx{0&I\\I&0}.$

The fixed points of the induced involution (also denoted by $\lambda$) on the Lie algebra $\fsp(2n,\C)$ gives the Lie algebra $\fsp(2n,\R)$ as the set of matrices $X=\smtrx{A&B\\C&-A^T}$ where $A$ is a $n\times n$ complex valued matrix with $A=-\overline A^T$ and $B$ is a complex valued $n\times n$ symmetric matrices with $C=\overline B$. 
Since the conjugation $\lambda$ commutes with the compact conjugation $g\ra \overline{g^{-1}}^T$ of $\sSp(2n,\C)$, the composition defines a Cartan involution $\theta$. On the Lie algebra $\fsp(2n,\R)$ the involution $\theta$ acts by 
\[\theta\left(\mtrx{A&B\\C&-A^T}\right)=\mtrx{A&-B\\-C&-A^T}.\] Thus, the Cartan decomposition is given by
\[\fsp(2n,\R)=\fh\oplus\fm=\fu(n)\oplus (sym(\R^n)\oplus sym(\R^n))\]
where $sym(\R^n)$ is the set of symmetric $n\times n$ real valued matrices. 
Complexifying this decomposition gives a decomposition of $\sH_\C=\sGL(n,\C)$-modules
\[
\fsp(2n,\C)=\fh_\C\oplus\fm_\C=\fgl(n,\C)\oplus Sym^n(V)\oplus Sym^n(V^*)\]
where $Sym^n(V)$ denotes the symmetric product of the standard representation of $\sGL(n,\C)$ on $\C^n.$
\begin{Definition}\label{DEF: Sp(2n,R) Higgs bundle}
		An {\em $\sSp(2n,\R)$-Higgs bundle} is a triple $(V,\beta,\gamma)$ where $V$ is a rank $n$ holomorphic vector bundle and $(\beta,\gamma)\in H^0(X,Sym^2(V)\otimes K)\oplus H^0(X,Sym^2(V^*)\otimes K)$.
\end{Definition}
Let $(V,\beta,\gamma)$ be an $\sSp(2n,\R)$-Higgs bundle, the holomorphic sections $\beta$ and $\gamma$ define holomorphic symmetric maps 
\[\xymatrix{\beta: V^*\to V \otimes K&\text{and}&\gamma:V \to V^*\otimes K}~.\]
The $\sSL(2n,\C)$-Higgs bundle associated to an $\sSp(2n,\R)$-Higgs bundle $(V,\beta,\gamma)$ is given by 
\[(E,\Phi)=\left(V\oplus V^*,\mtrx{0&\beta\\\gamma&0}\right).\]
The fundamental group of $\sU(n)$ is $\Z$, and the invariant $\omega\in\Z$ of an $\sSp(2n,\R)$-Higgs bundle $(V,\beta,\gamma)$ is the degree of the bundle $V.$ Moreover, polystability implies that $|deg(V)|\leq n(2g-2)$ \cite{sp4GothenConnComp}. Thus, the moduli space decomposes as 
\[\Mm(\sSp(2n,\R))=\bigsqcup\limits_{|\omega|\leq n(2g-2)}\Mm^{\omega}(\sSp(2n,\R)).\]
\smallskip

\noindent\textbf{$\sSO(p,q)$-Higgs bundles:\ }Fix $Q_p$ and $Q_q$ positive definite quadratic forms on $\R^p$ and $\R^q$ respectively and consider the signature $(p,q)$ form $Q=\smtrx{Q_p&\\&-Q_q}$ on $\R^{p+q}$. The group $\sSO(p,q)$ consists of matrices $g\in\sGL({p+q},\R)$ such that $g^TQg=Q.$ The group $\sSO(p,q)$ has two connected components, and the connected component of the identity will be denoted by $\sSO_0(p,q).$ 

The Lie algebra $\fso(p,q)$ consists of matrices $X$ such that $X^TQ+QX=0.$ 
A matrix $X\in\fso(p,q)$ decomposes as $\smtrx{A&-Q_q^{-1}B^TQ_p\\B&C}$, where $B$ is a $p\times q$ matrix, and $A$ and $C$ are respectively $p\times p$ and $q\times q$ matrices which satisfy
\[\xymatrix{A^TQ_p+Q_pA=0&\text{and}&C^TQ_q+Q_pC=0}.\] Thus, the Cartan decomposition is given by
\[\fso(p,q)=\fh\oplus\fm=(\fso(p)\oplus\fso(q))\oplus\Hom(\R^p,\R^q).\] 
Complexifying this decomposition gives a decomposition of $\sH_\C=\sS(\sO(p,\C)\times\sO(q,\C))$-modules\footnote{Note that this splitting is also preserved by $\sH_{\C,0}=\sSO(p,\C)\times\sSO(q,\C).$}
	\[\fso({p+q},\C)=\fh_\C\oplus\fm_\C=(\fso(p,\C)\oplus\fso(q,\C))\oplus\Hom(V,W)\]
where $V$ and $W$ denote the standard representations of $\sSO(p,\C)$ and $\sSO(q,\C)$ respectively. 
\begin{Definition}\label{DEF SO(p,q) Higgs bundles}
  	An $\sSO(p,q)$-{\em Higgs bundle} is a triple $(V,W,\eta)$ where $V$ and $W$ are holomorphic orthogonal bundles of rank $p$ and $q$ respectively such that $\Lambda^pV=\Lambda^qW,$ and 
  	\[\eta\in H^0(X,\Hom(V,W)\otimes K)~.\] An $\sSO_0(p,q)$-Higgs bundle is an $\sSO(p,q)$-Higgs bundle $(V,W,\eta)$ such that $\Lambda^pV=\Lambda^qW=\Oo.$
  \end{Definition}  
Let $(V,W,\eta)$ be an $\sSO(p,q)$-Higgs bundle, the orthogonal structures on $V$ and $W$ are holomorphic sections of $\Sym^2V$ and $\Sym^2W$ which define holomorphic symmetric isomorphisms 
\[\xymatrix{Q_V:V\to V^*&\text{and}&Q_W:W\to W^*}.\]
The $\sSL(p+q,\C)$-Higgs bundle associated to an $\sSO(p,q)$-Higgs bundles is given by 
\begin{equation}
	\label{EQ SL(p+q)Higgs ass SO(pq)}(E,\Phi)=\left(V\oplus W,\mtrx{0&\eta^\dagger\\\eta&0}\right)
\end{equation}
where $\eta^\dagger=-Q_q^{-1}\circ\eta\circ Q_p\in H^0(\Hom(W,V)\otimes K).$

For $2<p\leq q,$ the fundamental group of $\sSO_0(p,q)$ is $\Z_2\times\Z_2$, and the invariant $\omega\in\Z_2\times\Z_2$ of an $\sSO_0p,q)$-Higgs bundle $(V,W,\eta)$ is given by the second Stiefel-Whitney classes $(sw_2^p,sw_2^q)$ of $V$ and $W.$ Thus, the moduli space decomposes as 
\[\Mm(\sSO_0(p,q))=\bigsqcup\limits_{(sw_2^p, sw_2^q)}\Mm^{sw_2^p,sw_2^q}(\sSO_0(p,q)).\]
For a $\sSO(p,q)$-Higgs bundle $(V,W,\eta)$, the first Stiefel-Whitney class of $\Lambda^pV=\Lambda^qW$ defines another invariant and  
$\Mm(\sSO(p,q))=\bigsqcup\limits_{(sw_1,sw_2^p, sw_2^q)}\Mm_{sw_1}^{sw_2^p,sw_2^q}(\sSO(p,q)).$

\section{The $\sP\sGL(2,\R)=\sSO(1,2)$-character variety}\label{Section PGL2}For $\sSO(1,2),$ we can explicitly describe the Higgs bundle moduli space. Moreover, in this case, the connected component description is deduced from topological invariants of orthogonal bundles. We will see in later sections how these descriptions generalize to higher rank generalizations of $\sSO_0(1,2)=\sP\sSL(2,\R)=\sP\sSp(2,\R).$

Using Definition \ref{DEF SO(p,q) Higgs bundles}, an $\sSO(1,2)$-Higgs bundle $(V,W,\eta)$ is given by $(\Lambda^2 W,W,\eta)$ where $W$ is a rank two holomorphic vector bundle with an orthogonal structure $Q_W$. The $\sSL(3,\C)$-Higgs bundle associated to $(\Lambda^2 W,W,\eta)$ is given by \eqref{EQ SL(p+q)Higgs ass SO(pq)} and will be represented schematically as
\begin{equation}
	\label{EQ SL3 Higgsbundle}	\xymatrix{\Lambda^2W\ar@/_.75pc/[r]_{\eta}&W\ar@/_.75pc/[l]_{\eta^\dagger}}~,
\end{equation}
where we have suppressed the twisting by $K$ from the notation.
The topological invariants of an orthogonal bundle on $X$ are a first and second Stiefel-Whitney class $(sw_1,sw_2)\in H^1(X,\Z_2)\oplus  H^2(X,\Z_2).$ 
If $\Mm^{sw_2}_{sw_1}(\sSO(1,2))$ is the moduli space of $\sSO(1,2)$-Higgs bundles consisting of triple $(\Lambda^2 W, W,\eta)$ where the first and second Stiefel-Whitney classes $W$ are $(sw_1,sw_2)$, then
\begin{equation}
	\label{EQ sw1 sw2 decomposition} \Mm(\sSO(1,2))=\bigsqcup\limits_{\substack{(sw_1,sw_2)\in\\ H^1(X,\Z_2)\oplus H^2(X,\Z_2)}}\Mm_{sw_1}^{sw_2}(\sSO(1,2)).
\end{equation}

If the first Stiefel-Whitney class of $W$ vanishes, then the structure group of $W$ reduces to $\sSO(2,\C)$. 
Since $\sSO(2,\C)\cong\C^*,$ the holomorphic orthogonal bundle $(W,Q_W)$ is isomorphic to 
\[\left(M\oplus M^{-1},\mtrx{0&1\\1&0}\right)\] where $M\in\Pic^d(X)$ is a degree $d$ holomorphic line bundle. 
In this case, the second Stiefel-Whitney class is given by the degree of $M$ mod $2,$ and the Higgs field $\eta$ decomposes as
 \[\eta=(\mu,\nu)\in H^0(M^{-1}K)\oplus H^0(MK).\] 
In this case, the associated $\sSL(3,\C)$-Higgs bundle given by
\begin{equation}
	\label{EQ SL3 Higgsbundle sw_1=0}
	\xymatrix{M\ar@/_.75pc/[r]_{\mu}&\Oo\ar@/_.75pc/[r]_{\mu}\ar@/_.75pc/[l]_{\nu}&M^{-1}\ar@/_.75pc/[l]_{\nu}}.
\end{equation} 
The following two propositions are immediate. 
\begin{Proposition}
	If $deg(M)>0$, then the $\sSO(1,2)$-Higgs bundle \eqref{EQ SL3 Higgsbundle sw_1=0} is polystable if and only if $\mu\neq0\in H^0(M^{-1}K).$ If $deg(M)<0$ then the $\sSO(1,2)$-Higgs bundle \eqref{EQ SL3 Higgsbundle sw_1=0} is polystable if and only if $\nu\neq0\in H^0(MK).$ Thus, $|deg(M)|\leq 2g-2.$
\end{Proposition}
\begin{Proposition}\label{Prop: SO(2,1) switching}
 	The $\sS(\sO(1,\C)\times\sO(2,\C))$ gauge transformation \begin{equation}\label{EQ switching iso 1,2}
	\mtrx{&&-1\\&-1&\\-1&&}:M\oplus\Oo\oplus M^{-1}\longrightarrow M^{-1}\oplus \Oo\oplus M
\end{equation}
defines an isomorphism between $(M,\mu,\nu)$ and $(M^{-1},\nu,\mu).$ Thus we may assume $deg(M)\geq0.$
 \end{Proposition} 
 Let $\Mm_d(\sSO(1,2))$ denote the moduli space of polystable $\sSO(1,2)$-Higgs bundles of the form \eqref{EQ SL3 Higgsbundle} with vanishing first Stiefel-Whitney class and $deg(M)=d.$ The moduli space $\Mm_{sw_1=0}(\sSO(1,2))$ decomposes as 
\begin{equation}
	\label{EQ M(SO(1,2)) decomp}
	\Mm_{sw_1=0}(\sSO(1,2))=\bigsqcup\limits_{0\leq d\leq 2g-2}\Mm_d(\sSO(1,2)).
\end{equation}
\begin{Remark}
Note that the switching isomorphism \eqref{EQ switching iso 1,2} is in the $\sS(\sO(1,\C)\times\sO(2,\C))$-gauge group but not the $\sSO(1,\C)\times\sSO(2,\C)$-gauge group. In fact, the moduli space $\Mm(\sSO_0(1,2))$ is a double cover of $\Mm(\sSO(1,2))$. The inverse image of the map $\Mm(\sSO_0(1,2))\to\Mm_d(\sSO(1,2))$ is connected when $d=0$ and consists of two isomorphic components if $d\neq0$.
\end{Remark}

Hitchin proved the following theorem for $\sP\sSL(2,\R)=\sSO_0(1,2)$.
\begin{Theorem}\label{THM: Hitchin SO(2,1)}(Theorem 10.8 \cite{selfduality}) Let $X$ be a Riemann surface of genus $g\geq2.$ For each integer $d\in (0,2g-2]$, the moduli space $\Mm_d(\sSO(1,2))$ is smooth and diffeomorphic to a rank $(d+g-1)$-vector bundle $\Ff_d$ over the $(2g-2-d)^{th}$-symmetric product $\Sym^{2g-2-d}(X).$
\end{Theorem}
\begin{proof}
	Let $\widetilde\Ff_d=\{(M,\mu,\nu)\ |\ M\in\Pic^d(X),\ \mu\in H^0(M^{-1}K)\setminus\{0\},\ \nu\in H^0(MK)\}$. By the above discussion, there is a surjective map $\widetilde \Ff_d\ra\Mm_d(\sSO(1,2))$ defined by sending $(M,\mu,\nu)$ to the isomorphism class of the Higgs bundle \eqref{EQ SL3 Higgsbundle}. 
	It is straight forward to check that two points $(M,\mu,\nu)$ and $(M',\mu',\nu')$ in $\widetilde\Ff_d$ lie in the same $\sS(\sO(1,\C)\times\sO(2,\C))$-gauge orbit if and only if $M'=M,$ $\mu'=\lambda\mu$ and $\nu'=\lambda^{-1}\nu$ for $\lambda\in\C^*$.
	This gives a diffeomorphism between the quotient space $\Ff_d=\widetilde\Ff_d/\C^*$ and the moduli space $\Mm_d(\sSO(1,2))$. 
	The map $\pi_d:\Ff_d\ra\Sym^{2g-2-d}(X)$ defined by taking the projective class of $\mu$ is surjective. 
	For a divisor $D\in\Sym^{2g-2-d}(X),$ the fiber $\pi^{-1}(D)$ is identified (non-canonically) with $H^0(\Oo(-D)K^2)\cong\C^{d+g-1}$ where $\Oo(-D)$ is the inverse of the line bundle associated to $D.$
\end{proof}
\begin{Remark}\label{Remark: Fuchsian Higgs bundles}
	Note that when $d$ is maximal, the moduli space $\Mm_{2g-2}(\sSO(1,2))$ is diffeomorphic to the vector space $H^0(K^2)$ of holomorphic quadratic differentials on $X.$ Indeed, in this case, $\mu\in H^0(M^{-1}K)\setminus \{0\}$, so $M=K$ and $\nu\in H^0(K^2).$ The associated connected component $\Xx_{2g-2}(\sSO(1,2))$ is the set of Fuchsian representations from Example \ref{EX: Fuchsian reps}. 
\end{Remark}

\begin{Theorem}\label{THM M0 homotopy type}
	The space $\Mm_0(\sSO(1,2))$ retracts onto $\Pic^0(X)/\Z_2$ where $\Z_2$ acts by inversion.
\end{Theorem}
\begin{proof}
	Let $(M,\mu,\nu)$ be an $\sSO(1,2)$-Higgs bundle with $deg(M)=0.$ The associated $\sSL(3,\C)$-Higgs $(E,\Phi)$ bundle is given by \eqref{EQ SL3 Higgsbundle}. 
	Note that the bundle $E=M\oplus \Oo \oplus M^{-1}$ is polystable as a holomorphic vector bundle. Thus, the one parameter family $(E,t\Phi)$ has $\lim\limits_{t\to0}(E,t\Phi)=(E,0).$ 
	In terms of the data $(M,\mu,\nu)$ this one parameter family is given by $(M,t\mu,t\nu)$. 
	The moduli space hence deformation retracts onto the space $\Pic^0(X)/\Z_2$ where $\Z_2$ acts by inversion via the gauge transformation \eqref{EQ switching iso 1,2}.
\end{proof}
\begin{Remark}
	Note that the connected components $\Mm_d(\sSO(1,2))$ can be deformed to each other in $\Mm(\sSO_0(1,3))$. Indeed, $\sSO_0(1,3)\cong\sP\sSL(2,\C)$ and $\Mm(\sP\sSL(2,\C))$ has only two connected components which are distinguished by a second Stiefel-Whitney class. In particular, $\Mm_d(\sSO(1,2))$ can be deformed to $\Mm_{d'}(\sSO(1,2))$ inside $\Mm(\sSO_0(1,3))$ if and only if $d=d'\ \text{mod}\ 2.$ 
\end{Remark}
So far we have assumed that the first Stiefel-Whitney class of the $\sO(2,\C)$-bundle $W$ is zero. 
Equivalently, we have only considered $\sSO(1,2)$-Higgs bundles which reduce to $\sSO_0(1,2)$-Higgs bundles. 
We now recall Mumford's description of holomorphic $\sO(2,\C)$-bundles \cite{MumO2Bun}. 

Let $W\ra X$ be a holomorphic rank two vector bundle equipped with an orthogonal structure $Q_W$ with nonzero first Stiefel-Whitney class $sw_1\in H^1(X,\Z_2)\setminus\{0\}.$ 
Let $\pi:X_{sw_1}\ra X$ denote the corresponding connected orientation double cover associated to $sw_1.$ Note that $\pi^*(det(W))=\Oo_{X_{sw_1}}.$ Let $\iota:X_{sw_1}\ra X_{sw_1}$ denote the covering involution, and consider the space 
\begin{equation}
	\label{Eq Prym def}
	\Prym(X_{sw_1},X)=\{M\in\Pic^0(X_{sw_1})\ |\ \iota^*M=M^{-1}\}\subset \Pic^0(X_{sw_1})
\end{equation}
\begin{Proposition}\label{Prop Mumford O2} For $sw_1\in H^1( X,\Z_2)\setminus\{0\}$ let $\pi:X_{sw_1}\ra X$ be the corresponding unramified double cover, and denote the covering involution by $\iota:X_{sw_1}\ra X_{sw_1}$.
There is a bijection between $\Prym(X_{sw_1},X)$ and holomorphic $\sO(2,\C)$-bundles on $ X$ with first Stiefel-Whitney class $sw_1.$
\end{Proposition}
\begin{proof}
Let $(W,Q_W)$ be a holomorphic $\sO(2,\C)$-bundle on $ X.$ Since $X_{sw_1}$ is the orientation double cover, $sw_1(\pi^*W,\pi^*Q_W)=0$. 
Thus,
\[(\pi^*W,\pi^*Q_W)=\left(M\oplus M^{-1},\mtrx{0&1\\1&0}\right)\] and $(\pi^*W,\pi^*Q_W)$ is invariant under the covering involution 
	\[\iota^*(M\oplus M^{-1})=\iota^*M\oplus\iota^* M^{-1}\cong M\oplus M^{-1}.\]
Given $M\in Pic^0(X_{sw_1})$ with $\iota^*M=M^{-1}$, $(W,Q_W)=(\pi_*M,\pi_*\iota^*)$ is an orthogonal bundle. 
Since $X_{sw_1}\ra X$ is unramified, $\pi^*\pi_*(M)=M\oplus \iota^*M$, and the above construction gives a bijection. 
\end{proof}
\begin{Remark}
	There are two connected components of $\Prym(X_{sw_1},X)$. For $M\in\Prym(X_{sw_1},X)$, the second Stiefel-Whitney class of the orthogonal bundle $\pi_*M$ distinguishes the connected component which contains $M$ \cite{MumO2Bun}. Therefore we will write 
	\begin{equation}\label{EQ prym decomp}
		\Prym(X_{sw_1},X)=\bigsqcup\limits_{sw_2\in H^2(X,\Z_2)}\Prym^{sw_2}(X_{sw_1},X)~.
	\end{equation}
	The connected component of the identity is the {\em Prym variety} of the cover $\pi:X_{sw_1}\to X.$
\end{Remark}
It is not hard to show that the holomorphic bundle $W\oplus \Lambda^2W$ of a polystable Higgs bundle of the form \eqref{EQ SL3 Higgsbundle} with nonzero first Stiefel-Whitney class is polystable as a vector bundle.
Furthermore, the $\sS(\sO(1,\C)\times\sO(2,\C))$-gauge transformation $(g_{\Lambda^2W},g_W)=(det(Q_W),Q_W)$ defines an isomorphism between $(W,\eta)$ and $(W^*,\eta^\dagger).$ 
Thus, as an analog of Theorem \ref{THM M0 homotopy type} we have:
\begin{Theorem}\label{THM Homotopy type sw1}
	 For $(sw_1,sw_2)\in (H^1(X,\Z_2)\setminus\{0\})\times H^2(X,\Z_2)$, the space $\Mm_{sw_1}^{sw_2}(\sSO(1,2)$ from \eqref{EQ sw1 sw2 decomposition} deformation retracts onto the space $\Prym^{sw_2}(X_{sw_1},X)/\Z_2$ where $\Z/2$ acts by inversion. 
\end{Theorem}

\begin{Corollary}
	Let $\Xx_d(\Gamma,\sSO(1,2))$ and $\Xx_{sw_1}^{sw_2}(\Gamma,\sSO(1,2))$ denote the connected components of the character variety associated to $\Mm_d(\sSO(1,2))$ and $\Mm_{sw_2}^{sw_1}(\sSO(1,2)).$ If $\rho\in\Xx(\Gamma,\sSO(1,2))$, then $\rho$ can be deformed to a representation with compact Zariski closure if and only if $\rho$ is in $\Xx_0(\sSO(1,2))$ or $\Xx_{sw_1}^{sw_2}(\sSO(1,2))$ with $sw_1\neq0.$
\end{Corollary}
\begin{proof}
	Recall that a representation $\rho:\Gamma\to\sSO(1,2)$ has compact Zariski closure if and only if the Higgs field of the corresponding Higgs bundle is identically zero. By the above theorems this only happens in the components $\Mm_0(\sSO(1,2))$ or $\Mm_{sw_1}^{sw_2}(\sSO(1,2))$ with $sw_1\neq0.$
\end{proof}
\section{Deforming into Hitchin representations and maximal representations }\label{section Hitchin maximal}
We now describe some generalizations and deformations of $\sSO(1,2)$ and $\sP\sSL(2,\R)$ representations into split real groups and groups of Hermitian type. Hitchin representations into split groups and maximal representations into Hermitian groups define important families of connected components of the character variety since they are the only known components that consist entirely of representations satisfying Labourie's Anosov property \cite{AnosovFlowsLabourie,MaxRepsAnosov}. Here we will show that, apart from maximal $\sSO_0(2,3)$-representations, all maximal representations in $\sSp(2n,\R)$ and $\sSO_0(2,n)$ and all Hitchin representations can be interpreted as deformation spaces of Fuchsian representations.

For $\sP\sSL(2,\R)=\sSO_0(1,2)$, by Remark \ref{Remark: Fuchsian Higgs bundles}, the Higgs bundles which give rise to Fuchsian representations are given by 
\[\xymatrix{K\ar@/_.5pc/[r]_1&\Oo\ar@/_.5pc/[r]_1\ar@/_.5pc/[l]_{q_2}&K^{-1}\ar@/_.5pc/[l]_{q_2}}.\]
Lifting such a Higgs bundle to an $\sSL(2,\R)$-Higgs bundle is determined by choosing a square root $K^\haf$ of $K.$ Namely, the $\sSL(2,\R)$-Higgs bundles are given by 
$(E,\Phi)=\left(K^\haf\oplus K^{-\haf},\mtrx{0&q_2\\1&0}\right).$
Indeed, the second symmetric product of such an $(E,\Phi)$ gives the $\sSO_0(1,2)$-Higgs bundle above. The $\sSp(2,\R)$-Higgs bundle associated to $(E,\Phi)$ is 
$(V,\beta,\gamma)=(K^\haf,q_2,1). $
We will refer to such Higgs bundles as {\em Fuchsian Higgs bundles}.

\subsection{The Hitchin component}
Let $\sG$ be a split real Lie group, the classical split Lie groups are $\sP\sSL(n,\R),$ $\sSO_0(n,n+1)$, $\sP\sSp(2n,\R)$ and $\sP\sSO_0(n,n).$ For such a group, Kostant \cite{ptds} showed that there exists a special embedding of $\sP\sSL(2,\R)$ into $\sG$ called the principal embedding, for details on the principal embedding see section 3 of \cite{liegroupsteichmuller}. 
The principal embedding defines an ``irreducible'' way of deforming $\sP\sSL(2,\R)$-representations into $\Xx(\Gamma,\sG).$ 
When $\sG=\sP\sSL(n,\R)$ this embedding comes from the unique $n$-dimensional irreducible representation of $\sP\sSL(2,\R),$ namely the $(n-1)^{st}$ symmetric product of the standard representation of $\sP\sSL(2,\R).$

\begin{Definition}
	For a split real Lie group $\sG,$ the {\em Hitchin component} $\Hit(\sG)$ is the connected component of $\Xx(\Gamma,\sG)$ which contains $\iota\circ\rho_{Fuch}$ for $\rho_{Fuch}:\Gamma\to\sP\sSL(2,\R)$ a Fuchsian representation and $\iota:\sP\sSL(2,\R)\ra\sG$ the principal embedding. 
\end{Definition}

\begin{Theorem}(\cite{liegroupsteichmuller})
	The Hitchin component $\Hit(\sG)$ is smooth and diffeomorphic to a vector space of dimension $dim(\sG)(2g-2)$. Moreover, the Hitchin component does not contain representations with compact Zariski closure. 
\end{Theorem}
\begin{Remark}
	For the classical groups, the Hitchin component is parameterized as follows:
\[\xymatrix@R=.2em{\Hit(\sP\sSL(n,\R))\cong\bigoplus\limits_{j=2}^n H^0(K^j)&,& \Hit(\sSO_0(n,n+1))=\bigoplus\limits_{j=1}^nH^0(K^{2j})\\\Hit(\sP\sSp(2n,\R))\cong\bigoplus\limits_{j=1}^n H^0(K^{2j})&\text{and}& \Hit(\sP\sSO_0(n,n))=\bigoplus\limits_{j=1}^{n-1}H^0(K^{2j})\oplus H^0(K^{n})~.}\]
Hitchin representations into $\sP\sSL(2n,\R)$, $\sP\sSp(2n,\R)$ and $\sP\sSO_0(2n,2n)$ always lifts to $\sSL(2n,\R)$, $\sSp(2n,\R),$ and $\sSO_0(2n,2n).$ In all cases, there are $2^{2g}$ choices of lifts, and each choice defines a different connected component of the appropriate character variety.
\end{Remark}
\begin{Proposition}\label{DEF: SL(n,R) Hitchin component}
An $\sSL(n,\R)$-Higgs bundle $(E,\Phi)$ defines a point in a Hitchin component if it is gauge equivalent to 
	\[E=K^\frac{n-1}{2}\oplus K^\frac{n-3}{2}\oplus\cdots \oplus K^\frac{3-n}{2}\oplus K^\frac{1-n}{2}\] 
	and 
	\begin{equation}
		\label{EQ: SL(n,R) Hitchin component Higgs field}\Phi=\mtrx{0&q_2&q_3&\cdots&q_{n-1}&q_{n}\\
				 1&0&q_2&\cdots&q_{n-2}&q_{n-1}\\
				 &\ddots&\ddots&\ddots&&\\
				 &&1&0&q_2&q_3\\
				 &&&1&0&q_2\\
				 &&&&1&0}~:~E\longrightarrow E\otimes K ~.
	\end{equation}
			Here $q_{j}\in H^0(K^{j})$ and the orthogonal structure on $E$ is given by the pairing on each $K^j\oplus K^{-j}$.
\end{Proposition} 

\begin{Remark}
	Note that when $n$ is even, we have to choose a square root of the canonical bundle. The $2^{2g}$ components of $\Hit(\sSL(2n,\R))$ are given by twisting the bundle $E$ in Proposition \ref{DEF: SL(n,R) Hitchin component} by the $2^{2g}$ square roots of the trivial bundle. 
	Also, note that the zero locus of $q_2,\cdots,q_{n}$ is the $(n-1)^{st}$ symmetric product of the Fuchsian Higgs bundle $\big(K^\haf\oplus K^{-\haf},\smtrx{0&0\\1&0}\big)$. In particular, the representation $\Gamma\to\sSL(n,\R)$ corresponding to the zero locus of $q_2,\cdots,q_n$ is given by $\iota\circ\widehat\rho_{Fuch}$ where $\widehat\rho_{Fuch}:\Gamma\to\sSL(2,\R)$ is a lift of a Fuchsian representation and $\iota:\sSL(2,\R)\to\sSL(n,\R)$ is the principal embedding. In fact, it can be shown that the Fuchsian representation $\rho_{Fuch}$ uniformizes the Riemann surface $X.$ 
\end{Remark}
For the groups $\sSO_0(n,n+1)$ and $\sSp(2n,\R)$, the Hitchin component(s) can be seen as the subsets of the $\sSL(2n+1,\R)$ and $\sSL(2n,\R)$ Hitchin component(s) defined by the vanishing the differentials of odd degree. More precisely, we have the following proposition.
\begin{Proposition}\label{DEF SO(n,n+1) Hitchin comp}
An $\sSO_0(n,n+1)$-Higgs bundles $(V,W,\eta)$ is in the Hitchin component if 
\[\xymatrix{ V=K^{n-1}\oplus K^{n-3}\oplus\cdots\oplus K^{3-n}\oplus K^{1-n} & , & W=K^n\oplus K^{n-2}\oplus\cdots\oplus K^{2-n}\oplus K^{-n}}\]
and $\eta$ is the component of the Higgs field \eqref{EQ: SL(n,R) Hitchin component Higgs field} which maps $V$ to $W\otimes K.$

An $\sSp(2n,\R)$-Higgs bundle $(V,\beta,\gamma)$ is in a Hitchin component if, for a choice of $K^\haf$, 
\[V=K^\frac{2n-1}{2}\oplus K^\frac{2n-5}{2}\oplus\cdots\oplus K^\frac{7-2n}{2}\oplus K^\frac{3-2n}{2}\] 
and $\beta$ and $\gamma$ are the components of the Higgs field \eqref{EQ: SL(n,R) Hitchin component Higgs field} which map $V^*$ to $V\otimes K$ and $V$ to $V^*\otimes K.$
 \end{Proposition} 
The Hitchin component(s) for $\sSO_0(n,n)$ cannot be defined as a subspace of the $\sSL(2n,\R)$-Hitchin component. Rather, it is can be interpreted as the {\em deformation space} of $\sSO_0(n,n-1)$ Hitchin representations in the $\sSO_0(n,n)$-character variety\footnote{Indeed, one can show that the action of the principal embedding $\iota:\sP\sSL(2,\R)\to\sSO(n,n)$ on $\R^{2n}$ via the standard representations preserves a splitting $\R^{n,n}=\R^{n,n-1}\oplus\R^{0,1}.$}.
\begin{Proposition}\label{DEF: SO(n,n) Hitchin component}
	An $\sP\sSO_0(n,n)$-Higgs bundles $(V,W,\eta)$ is in the Hitchin component if 
	\[\xymatrix{W=W_0\oplus\Oo&\text{and}&\eta=(\eta_0~,~\eta_{Pf}):V\to W_0\otimes K\oplus K}\] are such that $(V,W_0,\eta_0)$ defines a Higgs bundles in the $\sSO_0(n,n-1)$-Hitchin component and $\eta_{Pf}:V\to K$ is given by 
	\[(0,0,\cdots,0,q_n):K^{n-1}\oplus K^{n-3}\oplus\cdots\oplus K^{3-n}\oplus K^{1-n}\longrightarrow K.\]
\end{Proposition}
\begin{Remark}
	When $n$ is odd, $\sSO_0(n,n)=\sP\sSO_0(n,n)$ and there is only one $\sSO_0(n,n)$-Hitchin component. However, when $n$ is even, $\sSO_0(n,n)$ is a double cover of $\sP\sSO_0(n,n).$ In this case, there are $2^{2g}$ connected components of $\Mm(\sSO_0(n,n))$ which map to the Hitchin component of $\sP\sSO_0(n,n).
	$ 
	These $2^{2g}$ components are distinguished by twisting both $V$ and $W$ from Proposition \ref{DEF: SO(n,n) Hitchin component} by one of the $2^{2g}$ square roots of the trivial bundle. 
\end{Remark}
Hitchin proved the following theorem concerning the connected components of $\Xx(\Gamma,\sP\sSL(n,\R).$ 
\begin{Theorem}
	(\cite{liegroupsteichmuller}) If $\rho\in\Xx(\Gamma,\sP\sSL(n,\R))$ and $n>2$, then either $\rho$ is a Hitchin representation or it can be continuously deformed to a representation with compact Zariski closure. In particular, $\Xx(\Gamma,\sP\sSL(n,\R))$ has three connected components if $n$ is odd and six connected components which come in isomorphic pairs when $n$ is even. 
\end{Theorem}
Thus, for $\sP\sSL(n,\R)$ the connected components of $\Xx(\sP\sSL(n,\R))$ satisfy the following dichotomy. 
\begin{Corollary}\label{COR of Hitchin's PSLnR THM}
	If $\rho\in\Xx(\sP\sSL(n,\R))$ with $n>2$, then we have the following dichotomy: 
	\begin{itemize}
		\item either $\rho$ can be continuously deformed to a representation with compact Zariski closure
		\item or $\rho$ can be continuously deformed to a representation  $\iota\circ\rho_{Fuch}$ where $\rho_{Fuch}:\Gamma\to\sP\sSL(2,\R)$ is a Fuchsian representations and $\iota:\sP\sSL(2,\R)\to\sP\sSL(n,\R)$ is the principal embedding.
		\end{itemize} 
\end{Corollary}
	For the other split groups such as $\sSO(n,n+1)$ and $\sSp(2n,\R)$ the situation is more complicated.

\subsection{Deforming into maximal representations}
We now describe how, for $n>2$, the set of maximal $\sSp(2n,\R)$ and $\sSO_0(2,n+1)$ representations can be realized as deformation spaces of Fuchsian representations. This follows from combining the work of \cite{TopInvariantsAnosov} with \cite{HiggsbundlesSP2nR} and \cite{HermitianTypeHiggsBGG}. Note that we have the following isomorphisms $\sP\sSp(2,\R)\cong\sSO_0(2,1)$ and $\sP\sSp(4,\R)\cong\sSO_0(2,3))$. 
Maximal $\sSO_0(2,1)$-representations are exactly Fuchsian representations, and maximal $\sSO_0(2,3)$ representations will be described in the next section. 

\medskip
\noindent\textbf{Maximal $\sSp(2n,\R)$:} Recall from Definition \ref{DEF: Sp(2n,R) Higgs bundle} that an $\sSp(2n,\R)$-Higgs bundle is given by a triple $(V,\beta,\gamma)$ where $V$ is a rank $n$ holomorphic vector bundle, $\beta\in H^0(Sym^2(V)\otimes K)$, and $\gamma\in H^0(Sym^2(V^*)\otimes K).$ Moreover, the Toledo invariant of $(V,\beta,\gamma)$ is given by the degree of $V.$ In this case, the Milnor-Wood $|deg(V)|\leq n(g-1).$ 

It is shown in \cite{sp4GothenConnComp} that, if $(V,\beta,\gamma)$ is a maximal $\sSp(2n,\R)$-Higgs bundle, then $\gamma:V\to V^*\otimes K$ is an isomorphism. Thus, $V\otimes K^{-\haf}$ is a holomorphic rank $n$ orthogonal bundle since
\[\gamma^*\circ\gamma: V\otimes K^{-\haf}\to V^*\otimes K^\haf\]
defines a symmetric isomorphism. The Stiefel-Whitney classes of $V\otimes K^{-\haf}$ give new topological invariants of maximal $\sSp(2n,\R)$-Higgs bundles. In \cite{HiggsbundlesSP2nR} it is shown that for $n>2$ the space of maximal $\sSp(2n,\R)$-Higgs bundles has $3\cdot 2^{2g}$ connected components, $2^{2g}$-Hitchin components and $2^{2g+1}$ components determined by Stiefel-Whitney classes. In particular, we have the following: 
\begin{Theorem}
	(\cite{HiggsbundlesSP2nR}) If $(V,\beta,\gamma)$ is a maximal $\sSp(2n,\R)$-Higgs bundle that is not in a Hitchin component and $n>2,$ then the first and second Stiefel-Whitney classes of $V\otimes K^{-\haf}$ uniquely determine the connected component of $(V,\beta,\gamma)$ in $\Mm(\sSp(2n,\R)).$
\end{Theorem}

Recall that if $\rho:\Gamma\to\sSL(2,\R)$ is a lift of a Fuchsian representation, then the associated $\sSp(2,\R)$-Higgs bundle is given by $(V,\beta,\gamma)=(K^{\haf},q_2,1)$ where $q_2$ is a holomorphic quadratic differential.
Consider the $\sSp(2n,\R)$-Higgs bundle which is the direct sum of $(K^{\haf},q_2,1)$ with itself $n$-times
\[(V,\beta,\gamma)=(K^\haf\oplus K^\haf\oplus\cdots \oplus K^\haf, q_2\oplus q_2\oplus\cdots\oplus q_2,1\oplus 1\oplus\cdots\oplus 1)~.\]
Note that $deg(V)=n(g-1),$ thus, this defines a maximal $\sSp(2n,\R)$-Higgs bundle. Moreover, if $I_1,\cdots, I_n$ are line bundles satisfying $I_j^2=\Oo$ the following $\sSp(2n,\R)$-Higgs bundle is also maximal
\begin{equation}\label{EQ: twisted Fuchsian}
	(V,\beta,\gamma)=(I_1K^\haf\oplus I_2K^\haf\oplus\cdots \oplus I_nK^\haf, q_2\oplus q_2\oplus\cdots\oplus q_2,1\oplus 1\oplus\cdots\oplus 1)~.
\end{equation}

The bundle $V$ in \eqref{EQ: twisted Fuchsian} satisfies $V\otimes K^{-\haf}=I_1\oplus\cdots\oplus I_n,$ in particular, $V\otimes K^{-\haf}$ is a rank $n$ orthogonal bundle. The total Stiefel-Whitney class of $V\otimes K^{-\haf}$ is given by 
\[1+sw_1(V\otimes K^{-\haf})+sw_2(V\otimes K^{-\haf})=1+\sum\limits_{j=1}^n sw_1(I_j)+\sum\limits_{j<k}sw_1(I_i)\smile sw_1(I_k)~.\] 
In particular, by varying the choices of $I_j$ we can obtain all possible values of the Stiefel-Whitney classes $(sw_1,sw_2)\in H^1(X,\Z_2)\oplus H^2(X,\Z_2).$ This proves the following:

\begin{Theorem}(\cite{TopInvariantsAnosov})
 If $\rho\in\Xx(\Gamma,\sSp(2n,\R))$ is a maximal representation and $n>2,$ then 
	\begin{itemize}
		\item either $\rho$ can be continuously deformed to $\rho_{Fuch}^1\oplus\cdots\oplus \rho_{Fuch}^n$ where $\rho^j_{Fuch}:\Gamma\to\sSL(2,\R)$ is a lift of a Fuchsian representation $\rho_{Fuch}:\Gamma\to\sP\sSL(2,\R)$.
		\item or $\rho$ can be continuously deformed to a lift of $\iota\circ\rho_{Fuch}$ where $\rho_{Fuch}:\Gamma\to\sP\sSL(2,\R)$ is a Fuchsian representation and $\iota:\sP\sSL(2,\R)\to\sP\sSp(2n,\R)$ is the principal embedding.
	\end{itemize}

	\end{Theorem}
\smallskip
\noindent\textbf{Maximal $\sSO_0(2,n)$:} We now focus on maximal $\sSO_0(2,n)$-representations. Recall from Definition \ref{DEF SO(p,q) Higgs bundles} that an $\sSO_0(2,n)$-Higgs bundle is given by a triple $(V,W,\eta)$ where 
\begin{itemize}
	\item $V=L\oplus L^{-1}$ is a rank 2 holomorphic bundle with orthogonal structure $Q_V=\smtrx{0&1\\1&0}:V\to V^*,$
	\item $W$ is a rank $n$ holomorphic bundle with orthogonal structure $Q_W:W\to W^*,$
	\item $\eta=(\beta,\gamma)\in H^0(\Hom(L^{-1},W)\otimes K)\oplus H^0(\Hom(L,W)\otimes K)$
\end{itemize}

The polystable $\sSL(n+2,\C)$ Higgs bundle $(E,\Phi)$ associated to $( L, W,Q_W,\beta,\gamma)$ 
has $E= L\oplus V\oplus L^{-1}$ and, using the notation from section \ref{Section PGL2}, $\Phi$ is given by  
\[\xymatrix{L\ar@/_.5pc/[r]_\gamma&W\ar@/_.5pc/[l]_{\beta^\dagger}\ar@/_.5pc/[r]_{\gamma^\dagger}&L^{-1}\ar@/_.5pc/[l]_\beta},\]
where, as before, we have suppressed the twisting by $K$ from the notation. 
The Toledo invariant of an $\sSO_0(2,n)$-Higgs bundle determined by $(L,W,\beta,\gamma)$ is given by the degree of $L.$ As with $\sSO_0(2,1)$, stability implies the Milnor-Wood inequality $deg(L)\leq 2g-2~.$
\begin{Proposition}\label{p:maximal Higgs bundle Param}
If $( L, W,\beta,\gamma)$ is a polystable $\sSO_0(2,n)$-Higgs bundle with $\deg( L)=2g-2,$ then $ L\cong K  I$ and $ W$ admits a $Q_W$-orthogonal decomposition $ W= I\oplus W_0$ where $ W_0$ is a holomorphic rank $n-1$ bundle and $ I=\Lambda^n W_0$ satisfies $ I^2=\Oo.$ Moreover, 
\[\xymatrix{\gamma\cong\mtrx{1\\0}:K I\to  I K\oplus  W_0\otimes K&\text{and}&\beta=\mtrx{q_2\\\beta_0}:K^{-1} I\to  I K\oplus W_0\otimes K}\] where $q_2\in H^0(K^2)$ and $\beta_0\in H^0(K\otimes I\otimes W_0).$

\end{Proposition}

\begin{proof}
If $deg(L)=2g-2$, then polystability implies $\gamma\neq0$ and the image of $\gamma$ is not contained in the kernel of $\gamma^T$. In particular, $\gamma^T\circ\gamma\in H^0((L^{-1}K)^2)\setminus\{0\}.$ 
This implies $( L^{-1}K)^2=\Oo$ and $\gamma$ is nowhere vanishing. Set $ I= L K^{-1}$, then $ L= I K$ and $ I$ defines an orthogonal line subbundle of $W$. 
Taking the $Q_W$-orthogonal complement of $ I$ gives a holomorphic decomposition $ W= I\oplus( I)^\perp$. 
Since $\Lambda^{n} W=\Oo,$ we conclude $ W= I\oplus W_0$ where $ I=\Lambda^n W_0.$
Since the image of $\gamma$ is identified with $ I,$ we can take $\gamma\cong\mtrx{1\\0}:K I\to  I K\oplus  W_0\otimes K$. 
Finally, the holomorphic section $\beta$ of $\Hom( I K^{-1}, I\oplus W_0)\otimes K$ decomposes as
$\beta= q_2\oplus \beta_0$
where $q_2\in H^0(K^2)$ and $\beta_0\in H^0( W_0\otimes I K)$.
\end{proof}

By the above proposition, maximal $\sSO_0(2,n)$-Higgs bundles is determined by a triple $(W_0,\beta_0,q_2)$ where $W_0$ is a rank $n-1$ orthogonal vector bundle. 
Let $\Mm_{sw_1}^{sw_2}(\sSO_0(2,n))$ denotes the space of maximal $\sSO_0(2,n)$-Higgs bundles such that the first and second Stiefel-Whitney classes of $W_0$ are $sw_1$ and $sw_2$; the space of maximal $\sSO_0(2,n)$-Higgs bundles decomposes as 
\begin{equation}
	\label{EQ max SO(2,n) sw decomp}
	\bigsqcup\limits_{\substack{(sw_1,sw_2)\in\\H^1(X,\Z_2)\oplus H^2(X,\Z_2)}}\Mm_{sw_1,sw_2}^{max}(\sSO_0(2,n))~.
\end{equation}
\begin{Remark}
	In \cite{HermitianTypeHiggsBGG} it is proven that, for $n>3$, the spaces $\Mm_{sw_1}^{sw_2}(\sSO_0(2,n))$ are nonempty and connected for each value of $(sw_1,sw_2)\in H^1(X,\Z_2)\oplus H^2(X,\Z_2)$. In particular, the space of maximal $\sSO_0(2,n)$-representations has $2^{2g+1}$ connected components for $n>3.$ 
\end{Remark}

We will now explain how each of the corresponding components of the character variety can be thought of as a deformation spaces of Fuchsian representations. 
Recall that if $\rho_{Fuch}:\Gamma\to\sSO_0(2,1)$ is a Fuchsian representation then the corresponding Higgs bundle is given by 
\[(L,\beta,\gamma)=(K,q_2,1).\]
If $W_0$ is a polystable rank $n-1$ orthogonal bundle with first Stiefel-Whitney class zero, then the $\sSO_0(2,n)$-Higgs bundle given by 
\[(L,W,\beta,\gamma)=(K ~,~ \Oo\oplus W_0~,~(q_2,0)~,~(1,0))\] 
is a maximal Higgs bundle in $\Mm_{sw_1=0}^{sw_2(W_0)}(\sSO_0(2,n))$. The associated representation is 
\[\rho=\rho_{Fuch}\oplus\alpha\] where $\alpha:\Gamma\to\sSO(n-1)$ is the representation associated to the polystable vector bundle $W_0.$ In particular, one can take $W_0=\Oo\oplus\cdots\oplus \Oo$, in which case $\alpha$ will be the trivial representation.

To obtain representations in the connected components with $sw_1\neq0$, consider a Fuchsian representation $\rho_{Fuch}:\Gamma\to\sSO_0(2,1)$ and let $\alpha:\Gamma\to\sO(n)$ be a representation so that the associated flat holomorphic $\sO(n-1,\C)$-bundle has first and second Stiefel-Whitney classes $sw_1$ and $sw_2.$ 
Denote the determinant representation of $\alpha$ by $\Lambda^n\alpha:\Gamma\to\sO(1).$ The Higgs bundle associated to the representation 
\[\rho_{Fuch}\otimes \Lambda^n\alpha\oplus\alpha\]
is 
\[(L,W,\beta,\gamma)=(KI, I\oplus W_0, (q_2,0), (1,0))\]
where the Higgs bundle associated to $\rho_{Fuch}$ is given by $(K,q_2,1)$ and $W_0$ is the flat holomorphic orthogonal bundle associated to $\alpha.$ In particular, the Higgs bundle is in $\Mm_{sw_1}^{sw_2}(\sSO_0(2,n)).$ 
\begin{Theorem}
 	If $n>3,$ then any maximal representation $\rho:\Gamma\to\sSO_0(2,n)$ can be continuously deformed to a representation 
 	\[\rho_{Fuch}\otimes \Lambda^n\alpha\oplus\alpha\]
 	where $\rho_{Fuch}:\Gamma\to\sSO_0(2,1)$ is a Fuchsian representation and $\alpha:\Gamma\to\sO(n-1).$ Moreover, the connected component of $\rho$ is determined by the Stiefel-Whitney classes of $\widetilde S\times_\alpha\sO(n-1).$
 \end{Theorem} 


\section{The special case of maximal $\sSO_0(2,3)\cong\sP\sSp(4,\R)$ representations}

The group $\sSp(4,\R)$ is a double cover of $\sSO_0(2,3)$. 
The case of maximal $\sSp(4,\R)$-Higgs bundles behave differently than the general case, similarly, maximal $\sSO_0(2,3)$-Higgs bundles behave differently than maximal $\sSO_0(2,n)$-Higgs bundles. We will focus on $\sSO_0(2,3)$ here since it generalizes $\sSO(2,1)$ and will be generalized in the next section. In particular, we will show that Theorems \ref{THM: Hitchin SO(2,1)}, \ref{THM M0 homotopy type} and \ref{THM Homotopy type sw1} for $\sSO(2,1)$-representations all generalize to maximal $\sSO_0(2,3)$-representations. 

Recall from Proposition \ref{p:maximal Higgs bundle Param}, that a maximal $\sSO_0(2,3)$-Higgs bundle is determined by a triple $(W_0,\beta_0,q_2)$ where $W$ is a holomorphic $\sO(2,\C)$-bundle with $\Lambda^2W_0=I$, $\beta_0\in H^0(W_0\otimes IK^2)$ and $q_2\in H^0(K^2).$ The corresponding $\sSL(5,\C)$-Higgs bundle can be represented schematically as
\begin{equation}
	\label{EQ max SO(2,3) schematic}
	\xymatrix@R=.5em{KI\ar[r]_1&I\ar[r]_1\ar@/_.5pc/[l]_{q_2}&K^{-1}I\ar@/_.5pc/[l]_{q_2}\ar@/^.3pc/[dl]^{\beta_0}\\&W_0\ar@/^.3pc/[ul]^{\beta_0^\dagger}&}~,
\end{equation}
where we again suppress the twisting by $K$ from the notation. When the first Stiefel-Whitney class of $W_0$ vanishes, the structure group of $W_0$ reduces to $\sSO(2,\C).$ In this case, $W_0$ is isomorphic to $M\oplus M^{-1}$ for some line bundle $M$ with $deg(M)\geq0$. 
Furthermore, the holomorphic section $\beta_0$ decomposes as $\beta_0=(\mu,\nu)\in H^0(M^{-1}K^2)\oplus H^0(MK^2).$ Schematically, we have
\begin{equation}
	\label{Eq  Gothen schematic}
	\xymatrix@C=4em@R=0em{KI\ar[r]_1&I\ar[r]_1\ar@/_.5pc/[l]_{q_2}&K^{-1}I\ar@/_.5pc/[l]_{q_2}\ar@/^.3pc/[dl]^{\nu}\ar@/^.5pc/[ddl]^\mu\\&M\ar@/^.3pc/[ul]^{\mu}&\\&M^{-1}\ar@/^.5pc/[uul]^\nu&}.
\end{equation}
If $deg(M)>0$, then such a Higgs bundle is polystable only if $\mu\neq0$. Thus, we have a bound \[0\leq \deg(M)\leq 4g-4~.\]

Analogous to the switching isomorphism from Proposition \ref{Prop: SO(2,1) switching} for $\sSO(2,1)$-Higgs bundles, the $\sSO(2,\C)\times\sSO(3,\C)$ gauge transformation $(g_{V},g_W)$ given by $g_V=-Id_{K\oplus K^{-1}}$ and
\[ g_W=\mtrx{&&-1\\&-1&\\-1&&}~:~ \xymatrix@C=1.5em{M\oplus \Oo\oplus M^{-1}\ar[r]&M^{-1}\oplus \Oo\oplus M}\]
defines an isomorphism between Higgs bundle associated to $(M,q_2,\mu,\nu)$ and the Higgs bundle associated to $(M^{-1},q_2,\nu,\mu).$ Thus we may assume $deg(M)\geq0.$ If $\Mm_{sw_1=0}^{max}(\sSO_0(2,3))$ denotes the space of maximal $\sSO_0(2,3)$-Higgs bundles with vanishing first Stiefel-Whitney class invariant, then we have the following decomposition analogous to \eqref{EQ M(SO(1,2)) decomp}: 
\[\Mm^{max}_{sw_1=0}(\sSO_0(2,3))=\bigsqcup\limits_{0\leq d\leq4g-4}\Mm^{max}_d(\sSO_0(2,3))\]
where $\Mm_d^{max}(\sSO_0(2,3))$ is the space of polystable maximal $\sSO_0(2,3)$-Higgs bundles given by tuples $(M,\mu,\nu,q_2)$ with $deg(M)=d.$
For $d>0$, the following generalization of Hitchin's theorem (Theorem \ref{THM: Hitchin SO(2,1)}) for the components $\Mm_d(\sSO_0(2,3))$ was proven in \cite{mythesis}. 
\begin{Theorem}\label{THM: Gothen component param}
	 For each integer $d\in (0,4g-4]$, the moduli space $\Mm^{max}_d(\sSO_0(2,3))$ is smooth and diffeomorphic to the product of a rank $(d+3g-3)$-vector bundle $\Ff_d$ over the $(4g-4-d)^{th}$-symmetric product $\Sym^{4g-4-d}(X)$ with the vector space $H^0(K^2).$
\end{Theorem}
The proof of above theorem is similar to that of Theorem \ref{THM: Hitchin SO(2,1)}. Namely, one considers the space $\widetilde\Ff_d\times H^0(K^2)$ where
\begin{equation}
	\label{EQ SO(2,3) tilde Fd}\widetilde\Ff_d=\{(M,\mu,\nu)\ |\ M\in\Pic^d(X),\ \mu\in H^0(M^{-1}K^2)\setminus\{0\},\ \nu\in H^0(MK^2)\}~.
\end{equation}
The map $\widetilde\Ff_d\times H^0(K^2)\longrightarrow \Mm_d^{max}(\sSO_0(2,3))$ given by sending a tuple $(M,\mu,\nu,q_2)$ to the isomorphism class of the $\sSO_0(2,3)$-Higgs bundle \eqref{Eq  Gothen schematic} is surjective. Moreover, one checks that two points $(M,\mu,\nu,q_2)$ and $(M',\mu',\nu',q_2')$ in $\widetilde\Ff_d\times H^0(K^2)$ lie in the same $\sSO(2,\C)\times\sSO(3,\C)$-gauge orbit if and only if $M'=M,$ $\mu'=\lambda\mu$,$\nu'=\lambda^{-1}\nu$ and $q_2'=q_2$ for $\lambda\in\C^*$. Now, the result follows just as in the case for $\sSO(2,1).$
\begin{Remark}
	When $d$ is maximal, the moduli space $\Mm^{max}_{4g-4}(\sSO_0(2,3))$ is diffeomorphic to the vector space $H^0(K^4)\times H^0(K^2)$. Since, $\mu\in H^0(M^{-1}K^2)\setminus \{0\}$, so $M=K^2$ and $\nu\in H^0(K^4).$ The associated connected component $\Xx^{max}_{4g-4}(\sSO_0(2,3))$ is the set of $\sSO_0(2,3)$-Hitchin representations. 
\end{Remark}
Theorem \ref{THM M0 homotopy type} and Theorem \ref{THM Homotopy type sw1} for $\Mm(\sSO(2,1))$ both generalize to maximal $\sSO_0(2,3)$ representations. Recall that if $sw_1\in H^1(X,\Z_2)\setminus \{0\}$ and $\pi:X_{sw_1}\to X$ is the corresponding orientation double cover, then the space $\Prym(X_{sw_1},X)\subset\Pic^0(X_{sw_1})$ defined in \eqref{Eq Prym def} has two connected components $\Prym^{sw_2}(X_{sw_1},X)$ labeled by a class $sw_2\in H^2(X,\Z/2)$. 

\begin{Theorem}\label{THM homotopy type singular SO(2,3)}
	Let $X$ be a Riemann surface with genus at least 2. The space of maximal $\sSO_0(2,3)$-Higgs bundles $\Mm^{max}_0(\sSO_0(2,3))$ with vanishing first Stiefel-Whitney class and $d=0$ described above deformation retracts onto $\Pic^0(X)/\Z_2$ where $\Z_2$ acts by inversion. Similarly, the space of maximal $\sSO_0(2,3)$-Higgs bundles $\Mm^{max}_{sw_1,sw_2}(\sSO_0(2,3))$ with  Stiefel-Whitney classes $(sw_1\neq0,sw_2)$ from \eqref{EQ max SO(2,n) sw decomp} deformation retracts onto $\Prym^{sw_2}(X_{sw_1},X)/\Z_2$ where $\Z_2$ acts by inversion. 
\end{Theorem}
\begin{Remark}
	Even though the spaces $\Mm^{max}_0(\sSO_0(2,3))$ and $\Mm_{sw_1,sw_2}^{max}(\sSO_0(2,3))$ are singular, one can still parameterize these singular spaces. Such parameterizations are of course much stronger results than the above theorem, and is carried out in \cite{SO23LabourieConj}.  
\end{Remark}
\begin{proof}
	Recall from Proposition \ref{p:maximal Higgs bundle Param} that a maximal $\sSO_0(2,3)$-Higgs bundle in one of the components $\Mm_{0}^{max}(\sSO_0(2,3))$ or $\Mm_{sw_1,sw_2}^{max}(\sSO_0(2,3))$ is given by a tuple $(W_0,q_2,\beta_0)$ where $W_0$ is an rank two orthogonal bundle with Stiefel-Whitney classes $sw_1$ and $sw_2$, $\beta_0\in H^0(W_0\otimes \Lambda^2W_0\otimes K)$ and $q_2\in H^0(K^2).$ Moreover, if $(W_0,q_2,\beta_0)$ defines a point in $\Mm_0^{max}(\sSO_0(2,3))$ then $W_0=M\oplus M^{-1}$ for some degree $0$ line bundle $M$. Note that, as in section \ref{Section PGL2}, $W_0$ is a polystable vector bundle. 

	The $\sSL(5,\C)$-Higgs bundle associated to a tuple $(W_0,q_2,\beta_0)$ is 
	\[(E,\Phi)=\left(IK\oplus IK^{-1}\oplus I\oplus W_0,\mtrx{ 0&0&q_2&\beta_0^\dagger\\0&0&1&0\\1&q_2&0&0\\0&\beta_0&0&0}\right).\]
	Consider the one parameter family of maximal $\sSO_0(2,3)$-Higgs bundles associated to $(E,t\Phi)$. It is straight forward to check that the $\sSO(2,\C)\times\sSO(3,\C)$ gauge transformation $g_t=\left(\smtrx{t&0\\0&t^{-1}} , \smtrx{1&0\\0&Id}\right)$ of $(KI\oplus K^{-1}I, I\oplus W_0)$ acts on $(E,t\Phi)$ by 
	\[g_t\cdot \left(IK\oplus IK^{-1}\oplus I\oplus W_0,\mtrx{ 0&0&tq_2&t\beta_0^\dagger\\0&0&t&0\\t&tq_2&0&0\\0&t\beta_0&0&0}\right)~=~\mtrx{ 0&0&t^2q_2&t^2\beta_0^\dagger\\0&0&1&0\\1&t^2q_2&0&0\\0&t^2\beta_0&0&0}. \]
	Since $W_0$ is a polystable, the $\sSO_0(2,3)$-Higgs bundle associated to $\lim\limits_{t\to 0}(E,t\Phi)$ is given by $(W_0,0,0).$ 

	If $(W_0,q_2,\beta_0)$ defines a point in $\Mm_0^{max}(\sSO_0(2,3))$ then $W_0=M\oplus M^{-1}$ for a $M\in\Pic^0(X).$ However, as in the proof of Theorem \ref{THM M0 homotopy type}, one cannot distinguish between $M$ and $M^{-1},$ and we conclude that the space $\Mm^{max}_0(\sSO_0(2,3))$ deformation retracts onto $\Pic^0(X)/\Z_2$ where $\Z/2$ acts by inversion. 
	Similarly, when $(W_0,q_2,\beta_0)$ defines a point in $\Mm_{sw_1,sw_2}^{max}(\sSO_0(2,3))$, $W_0$ defines a point in $\Prym(X_{sw_1}^{sw_2}(X_{sw_1},X)$. However, one cannot distinguish between $W_0$ and $W_0^*,$ and we conclude that the space $\Mm^{max}_{sw_1,sw_2}(\sSO_0(2,3))$ deformation retracts onto $\Prym^{sw_2}_{sw_1}(X_{sw_1},X)/\Z_2$.
\end{proof}
Let $\Xx_d^{max}(\Gamma,\sSO_0(2,3))$ and $\Xx_{sw_1,sw_2}^{max}(\Gamma,\sSO_0(2,3))$ be the connected components of the character variety associated to $\Mm_d^{max}(\sSO_0(2,3))$ and $\Mm_{sw_1,sw_2}^{max}(\sSO_0(2,3))$ respectively. The components  
\[\Xx^{max}_0(\Gamma,\sSO_0(2,3))~\sqcup~\Xx^{max}_{4g-4}(\Gamma,\sSO_0(2,3))~\sqcup~\bigsqcup\limits_{\substack{(sw_1\neq0,sw_2)\\\in H^1(X,\Z_2)\oplus H^2(X,\Z_2)}}\Xx_{sw_1,sw_2}^{max}(\Gamma,\sSO_0(2,3))\]
can be thought as deformation spaces of Fuchsian representations, while the remaining connected components $\bigsqcup\limits_{0<d<4g-4}\Xx_d^{max}(\Gamma,\sSO_0(2,3))$ cannot. More precisely, we have the following result which generalizes results of \cite{MaximalSP4} concerning maximal $\sSp(4,\R)$-representations.
\begin{Proposition} (\cite{SO23LabourieConj})\label{Prop: ZariskiClosures of SO23} For $\rho\in\Xx^{max}(\Gamma,\sSO_0(2,3))$ we have the following trichotomy:
\begin{itemize}
	\item If $\rho\in\bigsqcup\limits_{0<d<4g-4}\Xx_d^{max}(\Gamma,\sSO_0(2,3)),$ then $\rho$ is Zariski dense. 
	\item If $\rho\in\Xx_{4g-4}^{max}(\Gamma,\sSO_0(2,3))$, then $\rho$ can be continuously deformed to a representation $\iota\circ\rho_{Fuch}$ where $\rho_{Fuch}:\Gamma\to\sSO_0(2,1)$ is a Fuchsian representation and $\iota:\sSO_0(2,1)\to\sSO_0(2,3)$ is the principal embedding.
	\item Otherwise $\rho\in\Xx_0^{max}(\Gamma,\sSO_0(2,3))\sqcup\Xx_{sw_1,sw_2}^{max}(\Gamma,\sSO_0(2,3))$ and $\rho$ can be continuously deformed to a representation $\rho_{Fuch}\otimes \Lambda^2\alpha\oplus\alpha$ where $\rho_{Fuch}:\Gamma\to\sSO_0(2,1)$ is a Fuchsian representation and $\alpha:\Gamma\to\sO(2).$ Moreover, the Stiefel-Whitney classes of the flat $\sO(2)$-bundle associated to $\alpha$ are the same as the Stiefel-Whitney class invariants of $\rho.$
\end{itemize}
\end{Proposition}
Note that the second point follows since $\Xx_{4g-4}^{max}(\Gamma,\sSO_0(2,3))$ is the Hitchin component. The third point is a direct corollary of Theorem \ref{THM homotopy type singular SO(2,3)}.

\subsection{Deforming maximal $\sSO_0(2,3)$ into $\sSO_0(2,n)$ and $\sSO_0(3,3)$}
As we have seen, the set of maximal $\sSO_0(2,n)$-representations behave differently when $n=3$ compared to $n>3.$ However, for $n>3,$ consider the embedding of $i:\sSO_0(2,3)\to \sSO_0(2,n)$ given by the isometric embedding
\[\xymatrix@R=0em{\R^{2,3}\ar[r]&\R^{2,n}\\(x_1,\cdots,x_5)\ar@{|->}[r]& (x_1,\cdots,x_5,0\cdots,0)}~.\]
Using the notation \eqref{EQ max SO(2,3) schematic}, the induced map from maximal $\sSO_0(2,3)$ Higgs bundles to maximal $\sSO_0(2,n)$-Higgs bundles is given by 
\[\xymatrix@R=.5em{KI\ar[r]_1&I\ar[r]_1\ar@/_.5pc/[l]_{q_2}&K^{-1}I\ar@/_.5pc/[l]_{q_2}\ar@/^.3pc/[dl]^{\beta_0}\\&W_0\ar@/^.3pc/[ul]^{\beta_0^\dagger}&}\xymatrix{\ar@{|->}[rrr]^i&&&}
\xymatrix@C=4em@R=0em{KI\ar[r]_1&I\ar[r]_1\ar@/_.5pc/[l]_{q_2}&K^{-1}I\ar@/_.5pc/[l]_{q_2}\ar@/^.3pc/[dl]^{\beta_0}\\&W_0\ar@/^.3pc/[ul]^{\beta_0^\dagger}&\\&U&}\]
where $U$ is the direct sum of $n-4$ trivial bundles.
In particular, the space $\Mm_{sw_1,sw_2}^{max}(\sSO_0(2,3))$ is mapped into $\Mm_{sw_1,sw_2}^{max}(\sSO_0(2,4)).$ Since the Stiefel-Whitney class invariants $(sw_1,sw_2)$ determine the connected components of $\Mm_{sw_1,sw_2}^{max}(\sSO_0(2,4))$ the components $\Mm_d^{max}(\sSO_0(2,3))$ can be deformed to each other inside $\Mm^{max}(\sSO_0(2,4)).$ More precisely, we have the following. 
\begin{Proposition}\label{Prop: deform so23}
	Let $i:\Mm^{max}(\sSO_0(2,3))\to \Mm^{max}(\sSO_0(2,n))$ be the map induced by the isometric embedding $\R^{2,3}\to\R^{2,n+1}$. For $d$ even, the image of the components $\Mm^{max}_d(\sSO_0(2,3))$ under $i$ are all contained in $\Mm^{max}_{sw_1=0,sw_2=0}(\sSO_0(2,n)),$ while for $d$ odd, the image of the components $\Mm^{max}_d(\sSO_0(2,3))$ under $i$ are all contained in $\Mm^{max}_{sw_1=0,sw_2\neq0}(\sSO_0(2,n)).$   
\end{Proposition}
In terms of representations, this says that, for any $n>2,$ every maximal-$\sSO_0(2,3)$ can be continuously deformed in the set of maximal $\sSO_0(2,n)$ to a Fuchsian representation.
On the other hand, if $j:\sSO_0(2,3)\to\sSO_0(3,3)$ is the embedding given by the isometric embedding 
\[\xymatrix@R=0em{\R^{2,3}\ar[r]&\R^{3,3}\\(x_1,\cdots,x_5)\ar@{|->}[r]& (0,x_1,\cdots,x_5)}~,\]
we have the following:

\begin{Proposition}
Let $\rho:\Gamma\to\sSO_0(2,3)$ be a maximal representation and let $j:\Xx(\Gamma,\sSO_0(2,3))\to\Xx(\Gamma,\sSO_0(3,3))$ be the map induced by the embedding described above. If $\rho$ is a Hitchin representation in $\sSO_0(2,3)$, then $j(\rho)$ is a Hitchin representation in $\sSO_0(3,3).$ Otherwise, $j(\rho)$ can be continuously deformed to a representation with compact Zariski closure. 
\end{Proposition}
\begin{proof}
	By Proposition \ref{DEF: SO(n,n) Hitchin component}, the $\sSO_0(3,3)$-Hitchin component can be interpreted as the deformation space of the image of the $\sSO_0(2,3)$-Hitchin component under the map $j.$ For the second point, recall that the Lie groups $\sSO_0(3,3)$ and $\sP\sSL(4,\R)$ are isomorphic. Hence the result follows from Corollary \ref{COR of Hitchin's PSLnR THM} of Hitchin's theorem. 
\end{proof}

\section{Generalizing maximal $\sSO_0(2,3)$ representations to $\sSO_0(n,n+1)$ and $\sSO_0(n,n+2)$}\label{section DEFORMATIONS}

In the previous section we saw how many of Hitchin's results for the $\sSO(1,2)$-Higgs bundles had generalizations to the set of maximal $\sSO_0(2,3)$-Higgs bundles. These generalizations followed from the extra symmetries maximality imposed on the Higgs field. Moreover, using the low dimensional isomorphism $\sSO_0(2,3)=\sP\sSp(4,\R)$, we saw that the special features for maximal $\sSp(4,\R)$-Higgs bundles do not generalize to $\sSp(2n,\R).$ 
In this section, we will discuss some results from \cite{mythesis} and \cite{SOnn+1PosRepsHiggs} which show that the special features of the maximal $\sSO_0(2,3)$-Higgs bundles have generalizations in the space of $\sSO_0(n,n+1)$-Higgs bundles. In this case, there is no known ``topological invariant'' which distinguishes these generalizations. The following theorem is the main result.

\begin{Theorem}\label{THM SOnn+1comp}
	For each integer $d\in (0,n(2g-2)]$, there is a connected component 
	\[\Mm_d(\sSO_0(n,n+1))\subset\Mm(\sSO_0(n,n+1)\] which is smooth and diffeomorphic to the product of a rank $(d+(2n-1)(g-1))$-vector bundle $\Ff_d$ over the $(n(2g-2)-d)^{th}$-symmetric product $\Sym^{n(2g-2)-d}(X)$ with the vector space $\bigoplus\limits_{j=1}^{n-1}H^0(K^2).$
\end{Theorem}
\begin{Corollary}
	When the integer invariant $d$ is maximal (i.e. $d=n(2g-2)$), the above theorem recovers Hitchin's parameterization of the $\sSO_0(n,n+1)$-Hitchin component.
\end{Corollary}
\begin{Remark}
	 For $n\geq2,$ $\sSO_0(n,n+1)$ is not of Hermitian type and the topological invariants associated to an $\sSO_0(n,n+1)$-representation are two second Stiefel-Whitney classes. In particular, if $\Xx_d(\Gamma,\sSO_0(n,n+1))$ denotes the connected component of the character variety corresponding to $\Mm_d(\sSO_0(n,n+1),$ then these invariants do not distinguish the components $\Xx_d(\Gamma,\sSO_0(n,n+1))$. In fact, these are the first examples of non-Hitchin and non-maximal connected components of character varieties $\Xx(\Gamma,\sG)$ which are not distinguished by a topological invariant $\omega\in\pi_1(\sG).$ 
\end{Remark}
We start by considering a natural generalization of the space $\widetilde\Ff_d$ from \eqref{EQ SO(2,3) tilde Fd}. Consider the space 
\[\widetilde\Ff_d=\{(M,\mu,\nu)\ |\ M\in\Pic^d(X),\ \mu\in H^0(M^{-1}K^n)\setminus\{0\},\ \nu\in H^0(MK^n)\}~.\]
Note that the condition $\mu\in H^0(M^{-1}K^n)\setminus\{0\}$ implies that the integer $d$ satisfies the bound 
\[0\leq d\leq n(2g-2).\]

Associated to a point in $\widetilde\Ff_d\times\bigoplus\limits_{j=1}^{n-1}H^0(K^{2j})$ we can construct an $\sSO_0(n,n+1)$-Higgs bundles by  
\[(M,\mu,\nu,q_2,\cdots,q_{2n-2})\longrightarrow \left(V~,~W_0\oplus M\oplus M^{-1}~,~\mtrx{\eta_0\\\alpha\\\beta}\right)\]
where $(V,W_0,\eta_0)$ is the $\sSO_0(n-1,n)$-Higgs bundle in the Hitchin component associated to $(q_2,\cdots,q_{2n-2})$ (see Proposition \ref{DEF SO(n,n+1) Hitchin comp}) and 
\[\mtrx{\alpha\\\beta}:V=K^{n-1}\oplus K^{n-3}\oplus\cdots K^{3-n}\oplus K^{1-n}\longrightarrow (M\oplus M^{-1})\otimes K\]
is given by
\[\xymatrix{\alpha=(0,\cdots,0,\nu)&\text{and}&\beta=(0,\cdots,0,\mu)}~.\]
Moreover, one can show that, for $d>0$, such a Higgs bundle defines a {\em polystable} $\sSO_0(n,n+1)$-Higgs bundle. 
Hence, we have a map $\widetilde\Psi_d$  from the set $\widetilde\Ff_d\times\bigoplus\limits_{j=1}^{n-1}H^0(K^{2j})$ to the set of polystable $\sSO_0(n,n+1)$-Higgs bundles. The next step is to show that the only $\sSO(n,\C)\times\sSO(n+1,\C)$-gauge transformations which preserve the image of $\widetilde\Psi_d$ act by 
\[\widetilde\Psi_d(M,\mu,\nu,q_2,\cdots,q_{2n-2})\longrightarrow \widetilde\Psi_d(M,\lambda\mu,\lambda^{-1}\nu,q_2,\cdots,q_{2n-2})\] 
for $\lambda\in\C^*.$ This is done directly. 

The map $\widetilde\Psi_d$ therefore descends to a map
\[\Psi_d:\Ff_d\times (\bigoplus\limits_{j=1}^{n-1} H^0(K^{2j}))\longrightarrow \Mm(\sSO_0(n,n+1))\]
where $\Ff_d$ is the rank $(d+(2n-1)(g-1))$-vector bundle over $\Sym^{n(2g-2)-d}(X)$ given by 
\[\Ff_d=\widetilde\Ff_d/\C^*\]
where $\C^*$ acts by $\lambda\cdot(M,\mu,\nu)=(M,\lambda\mu,\lambda^{-1}\nu).$ 

For the $\sSO_0(1,2)$ and $\sSO_0(2,3)$ cases, we arrived at the above description via restrictions given by certain topological invariants. In the general case, we do not have these topological invariants, so we must show $\Psi_d$ is open and closed. 
To show that the image of $\Psi_d$ is closed in $\Mm(\sSO_0(n,n+1))$ we use the properness of the Hitchin fibration. 
Namely, if a sequence diverges in the parameter space $\Ff_d\times\bigoplus\limits_{j=1}^{n-1} H^0(K^{2j}),$ then the corresponding points in the Hitchin base associated to the image of $\Psi_d$ will also diverge. To finish the argument, note that, by a simple dimension count, image of $\Psi_d$ is the expect dimension of the moduli space. 

\begin{Remark}\label{REM M0 component SOn,n+1}
	In fact, it is shown in \cite{SOnn+1PosRepsHiggs} that all of the components of $\Mm(\sSO(1,2))$ generalize to $\sSO(n,n+1).$ In particular, there is also a connected component $\Mm_0(\sSO_0(n,n+1))$ of $\Mm(\sSO_0(n,n+1))$ which corresponds to the above integer $d$ being zero. 
	Moreover, if $\Xx_0(\Gamma,\sSO_0(n,n+1))$ is the connected component of the character variety associated to $\Mm_0(\sSO_0(n,n+1)$, then one can show that every representation $\rho\in\Xx_0(\Gamma,\sSO_0(n,n+1)$ {\em cannot} be deformed to a compact representation, but can be continuously deformed to a representation 
	\[(\iota\circ\rho_{Fuch})\oplus\alpha~.\]
	Here $\rho_{Fuch}:\Gamma\to\sP\sSL(2,\R)$ is a Fuchsian representation, $\iota:\sP\sSL(2,\R)\to\sSO_0(n,n-1)$ is the principal embedding and $\alpha:\Gamma\to\sSO(2).$ This should be interpreted as a result analogous to Proposition \ref{Prop: ZariskiClosures of SO23}. 
	Generalizing the notation of \eqref{EQ max SO(2,3) schematic}, the $\sSO_0(n,n+1)$-Higgs bundle associated $(\iota\circ\rho_{Fuch})\oplus\alpha$ is given by 
	\[\xymatrix@R=.3em{K^{n-1}\ar[r]_1&K^{n-2}\ar[r]_1&\cdots\ar[r]_1&K\ar[r]_1&I\ar[r]_1&K^{-1}\ar[r]_1&\cdots\ar[r]_1&K^{2-n}\ar[r]_1&K^{1-n}\\&&&&M&&&&\\&&&&M^{-1}&&&&}~,\]
	where $M$ is a holomorphic line bundle of degree zero. 
	
	Let $\Xx_d(\Gamma,\sSO_0(n,n+1))$ be the connected component of $\Xx(\Gamma,\sSO_0(n,n+1))$ which corresponds to $\Mm_d(\sSO_0(n,n+1)$. For representations in $\Xx_d(\Gamma,\sSO_0(n,n+1),$ there are no obvious model representations to deform to. It is most likely that, for $0<d<n(2g-2)$, all the representations in the components $\Xx_d(\Gamma,\sSO_0(n,n+1)$ are Zariski dense.
\end{Remark}

\subsection{Deforming $\sSO_0(n,n+1)$ into $\sSO_0(n,n+2)$}
To conclude, we show that, analogous to Proposition \ref{Prop: deform so23}, all of the connected components $\Mm_d(\sSO_0(n,n+1))$ described above can be deformed into each other in the space $\Mm(\sSO_0(n,n+2)).$ We will do this by constructing explicit deformations. 
The ideas in the analysis below are similar to how the Morse flow works in the $\sSO_0(1,3)=\sP\sSL(2,\C)$-Higgs bundles moduli space. These ideas play an essential role in \cite{SOpqStabilityAndMinima} where we describe the connected components of $\Mm(\sSO(p,q)).$ However, in the general $\sSO(p,q)$ case, the arguments become much more complex. For clarity and notational convenience, we will describe how this works in the case $n=3.$

Let $i:\sSO_0(3,4)\to\sSO_0(3,5)$ be the embedding induced by the isometric embedding 
\[\xymatrix@R=0em{\R^{3,4}\ar[r]^i&\R^{3,5}\\(x_1,\cdots,x_7)\ar@{|->}[r]&(x_1,\cdots,x_7,0)}~.\]
This induces a map from $\sSO_0(3,4)$-Higgs bundles to $\sSO_0(3,5)$-Higgs bundles given by sending $(V,W,\eta)$ to $(V,W\oplus\Oo,\smtrx{\eta\\0}).$ 
For a fixed $d>0,$ consider a Higgs bundle in the connected component $\Mm_d(\sSO_0(3,4))$ given by 
\begin{equation}
	\label{Eq Higgs field eta}
	\xymatrix@C=1em{V=K^2\oplus\Oo\oplus K^{-2}&W=M\oplus K\oplus K^{-1}\oplus M^{-1}&\eta=\smtrx{0&0&0\\1&0&0\\0&1&0\\0&0&\mu}:V\to W\otimes K}
\end{equation}
where $M\in\Pic^d(X)$ and $\mu\in H^0(M^{-1}K^3)\setminus\{0\}.$ 

For each $\epsilon\in\Omega^{0,1}(X,M^{-1})$ with $[\epsilon]\in H^1(M^{-1})\setminus\{0\}$, consider the holomorphic structure on the {\em smooth bundle} $M\oplus M^{-1}\oplus \Oo$ given by 
\[\bar\p_\epsilon=\mtrx{\bar\p_M&&\\&\bar\p_{M^{-1}}&\epsilon\\-\epsilon&&\bar\p_\Oo}\in\Omega^{0,1}(X,\End(M\oplus M^{-1}\oplus\Oo))~.\]

Note that the orthogonal structure $Q=\smtrx{0&1&0\\1&0&0\\0&0&1}$ on $M\oplus M^{-1}\oplus\Oo$ is holomorphic with respect to $\bar\p_\epsilon$. Note also that, if $\eta$ is given by \eqref{Eq Higgs field eta}, then the Higgs field $\smtrx{\eta\\0}:V\to W\oplus \Oo$ is holomorphic with respect to the holomorphic structures
\[\bar\p_V=\smtrx{\bar\p_{K^2}&&\\&\bar\p_\Oo&\\&&\bar\p_{K^{-2}}}\ \ \ \ \ \ \text{and}\ \ \ \ \ \ \bar\p^\epsilon_{W\oplus\Oo}=\smtrx{\bar\p_M&&&&\\&\bar\p_{K}&&&\\&&\bar\p_{K^{-1}}&&\\&&&\bar\p_{M^{-1}}&\epsilon\\-\epsilon&&&&\bar\p_\Oo }~. \]
Hence, $(\bar\p_V,\bar\p_{W\oplus\Oo}^\epsilon,\smtrx{\eta\\0})$ defines an $\sSO_0(3,5)$-Higgs bundle. Moreover, $(\bar\p_V,\bar\p_{W\oplus\Oo}^\epsilon,\smtrx{\eta\\0})$ is polystable, since any potentially destabilizing subbundle of $(V\oplus W\oplus\Oo)$ would also destabilize the original $\sSO_0(3,4)$-Higgs bundle \eqref{Eq Higgs field eta}. 

\begin{Proposition}\label{Prop deform nn+2}
For $0<d,$ let $M\in\Pic^d(X)$ and fix 
\[\xymatrix{[\epsilon]\in H^1(M^{-1})\setminus\{0\}&\text{and}&\mu\in H^0(M^{-1}K^3)\setminus\{0\}}~.\]
 If $(\bar\p_V,\bar\p_{W\oplus\Oo}^\epsilon,\smtrx{0\\\eta})$ is the $\sSO_0(3,5)$ described above, then:
\begin{itemize}
	\item  $\lim\limits_{t\to\infty}(\bar\p_V,\bar\p_{W\oplus\Oo}^\epsilon,t\smtrx{0\\\eta})$ exists in the moduli space and is given by \eqref{Eq Higgs field eta}.
	\item $\lim\limits_{t\to0}(\bar\p_V,\bar\p_{W\oplus\Oo}^\epsilon,t\smtrx{0\\\eta})$ is either given by 
	\begin{equation}
		\label{unstable case}
		\xymatrix@C=.8em{V'=K^2\oplus\Oo\oplus K^{-2}&W'=N\oplus K\oplus K^{-1}\oplus N^{-1}\oplus\Oo&\eta'=\smtrx{0&0&0\\1&0&0\\0&1&0\\0&0&\alpha\\0&0&0}:V\to W\otimes K}
	\end{equation}
	where $deg(N)\equiv d\ \text{mod}\ 2$ and $\alpha\in H^0(N^{-1}K^3)\setminus \{0\}$, or 
	\begin{equation}
		\label{stable case}
		\xymatrix{V'=K^2\oplus\Oo\oplus K^{-2}&W= K\oplus K^{-1}\oplus W_0&\eta'=\smtrx{1&0&0\\0&1&0\\0&0&0}:V\to W\otimes K}
	\end{equation}
	where $W_0$ is a polystable $\sSO(3,\C)$ bundle with second Stiefel-Whitney class $d\ \text{mod}\ 2.$
	\end{itemize}
\end{Proposition}
\begin{Remark}
	Note that this proposition implies that the connected components $\Mm_d(\sSO_0(3,4))$ can be deformed to $\Mm_{d'}(\sSO_0(3,4))$ inside $\Mm(\sSO_0(3,5))$ if and only if $d\equiv d'\ \text{mod}\ 2.$ 
\end{Remark}
We also have the following corollary.
\begin{Corollary}
	Every $\rho\in\Xx_d(\sSO_0(3,4))$ can be deformed in $\Xx(\sSO_0(3,5))$ to a representation $\iota\circ\rho_{Fuch}\oplus\alpha$ where $\rho_{Fuch}:\Gamma\to\sSO_0(2,1)$ is a Fuchsian representation, $\iota:\sSO_0(2,1)\to\sSO_0(3,2)$ is the principal embedding and $\alpha:\Gamma\to\sSO(3).$
\end{Corollary}
\begin{proof}(of Proposition \ref{Prop deform nn+2})
	For the limit as $t\to\infty$, note that the $\sSO(3,\C)\times\sSO(5,\C)$ gauge transformations 
	\[g^t_V=\smtrx{t^2&&\\&1&\\&&t^{-2}}\ \ \ \ \ \text{and}\ \ \ \ \ \ g^t_{W\oplus\Oo}=\smtrx{t^3&&&&\\&t^1&&&\\&&t^{-1}&&\\&&&t^{-3}&\\&&&&1}\]
	act on $\bar\p_{V},$ $\bar\p_{W\oplus\Oo}^\epsilon,$ and $t\smtrx{\eta\\0}$ by $(g^t_V,g^t_{W\oplus\Oo})\cdot \bar\p_{V}=\bar\p_V$ and
	\[(g^t_V,g^t_{W\oplus\Oo})\cdot \bar\p_{W\oplus\Oo}^\epsilon=\smtrx{\bar\p_M&&&&\\&\bar\p_{K}&&&\\&&\bar\p_{K^{-1}}&&\\&&&\bar\p_{M^{-1}}&t^{-3}\epsilon\\-t^{-3}\epsilon&&&&\bar\p_\Oo }\ \ \ \ \ \text{and} \  \ \ \  \ (g^t_V,g^t_{W\oplus\Oo})\cdot t\smtrx{\eta\\0}=\smtrx{\eta\\0}.\]
	After acting by this gauge transformation, it becomes clear that $\lim\limits_{t\to\infty}(\bar\p_V,\bar\p_{W\oplus\Oo}^\epsilon,t\smtrx{0\\\eta})$ is given by \eqref{Eq Higgs field eta}.

For the limit as $t\to0,$ let $W_0$ be the holomorphic orthogonal bundle with Dolbeault operator $\bar\p_\epsilon=\smtrx{\bar\p_M&&\\&\bar\p_{M^{-1}}&\epsilon\\-\epsilon&&\bar\p_\Oo}$ in the smooth splitting $M\oplus M^{-1}\oplus\Oo$. 

Suppose $W_0$ is a polystable holomorphic orthogonal bundle. The gauge transformations
\[g^t_V=\smtrx{t^2&&\\&1&\\&&t^{-2}}\ \ \ \ \ \text{and}\ \ \ \ \ \ g^t_{W\oplus\Oo}=\smtrx{1&&&&\\&t^1&&&\\&&t^{-1}&&\\&&&1&\\&&&&1}\]
act as $(g^t_V,g^t_{W\oplus\Oo})\cdot\bar\p_V=\bar\p_V$, $(g^t_V,g^t_{W\oplus\Oo})\cdot\bar\p_{W\oplus\Oo}^\epsilon=\bar\p_{W\oplus\Oo}^\epsilon,$ and 
\[(g^t_V,g^t_{W\oplus\Oo})\ \cdot\ t\smtrx{\eta\\0}=\smtrx{0&0&0\\1&0&0\\0&1&0\\0&0&t^3\mu\\0&0&0}~.\]
Since $W_0$ is assumed to be polystable, the limit as $t\to0$ is given by \eqref{stable case}.

Now assume $W_0$ is an unstable $\sSO(3,\C)$ bundle. In this case, $W_0$ has a unique destabilizing (positive degree) isotropic\footnote{One way to interpret this is that an $\sSO(3,\C)$ bundle with vanishing second Stiefel-Whitney class is the second symmetric product of rank $2$ holomorphic vector bundle $V$.} line subbundle $N\subset W_0.$ Thus, in the smooth splitting $N\oplus N^{-1}\oplus \Oo$ of $W_0$ we can write 
\[\bar\p_{W_0}^\epsilon= \mtrx{\bar\p_N&&-\delta\\&\bar\p_{N^{-1}}&\\&\delta&\bar\p_\Oo}\]
for some $[\delta]\in H^1(N)\setminus\{0\}.$ 
In the smooth splitting $W=N\oplus K\oplus K^{-1}\oplus N^{-1}\oplus\Oo,$ the Higgs field $t\smtrx{\eta\\0}$ is given by 
\begin{equation}
	\label{EQ N smooth splitting} \smtrx{0&0&t\beta\\t&0&0\\0&t&0\\0&0&t\alpha\\0&0&t\gamma}
\end{equation}
with $\alpha\neq0.$
Now, in the smooth splitting $W=N\oplus K\oplus K^{-1}\oplus N^{-1}\oplus\Oo,$ the gauge transformations 
\[g^t_V=\smtrx{t^2&&\\&1&\\&&t^{-2}}\ \ \ \ \ \text{and}\ \ \ \ \ \ g^t_{W\oplus\Oo}=\smtrx{t^3&&&&\\&t^1&&&\\&&t^{-1}&&\\&&&t^{-3}&\\&&&&1}\]
act by $(g^t_V,g^t_{W\oplus\Oo})~\cdot~\bar\p_V=\bar\p_V$, 
\[(g^t_V,g^t_{W\oplus\Oo})\cdot \bar\p_{W\oplus\Oo}^\epsilon=\smtrx{\bar\p_N&&&&-t^3\delta\\&\bar\p_{K}&&&\\&&\bar\p_{K^{-1}}&&\\&&&\bar\p_{N^{-1}}&\\&&&t^{3}\delta&\bar\p_\Oo }\ \ \ \ \ \text{and} \  \ \ \  \ (g^t_V,g^t_{W\oplus\Oo})\cdot t\smtrx{\eta\\0}=\smtrx{0&0&t^6\beta\\1&0&0\\0&1&0\\0&0&\alpha\\0&0&t^3\gamma}.\]
After changing by this gauge, it is clear that $\lim\limits_{t\to0}(\bar\p_V,\bar\p_{W\oplus\Oo}^\epsilon,t\smtrx{0\\\eta})$ is given by \eqref{unstable case}.
\end{proof}
 
 \begin{Remark}
 	In \cite{SOpqStabilityAndMinima}, Higgs bundles are used to show that the representations in the connected components $\Xx_d(\Gamma,\sSO_0(3,4))$ cannot be deformed to compact representations in the $\sSO_0(3,5)$ character variety. In particular, this implies the existence of {\em exotic} connected components of $\Xx(\Gamma,\sSO_0(3,5)).$ More generally, this is carried out for $\Xx(\Gamma,\sSO(p,q))$ when $2<p<q.$
 \end{Remark}
\bibliography{../../mybib}{}
\bibliographystyle{plain}
\end{document}

%% file: preamble.tex
\usepackage[T1]{fontenc}
\usepackage[utf8]{inputenc}
\usepackage[english]{babel}
\usepackage{xypic}
\input xy
\xyoption{all}
\usepackage{epsfig}
\usepackage{color}
\usepackage{amsthm}
\usepackage{changepage}
\usepackage{amssymb}
\usepackage{mathdots}
\usepackage{amsmath}
\usepackage{amscd}
\usepackage{amsopn}
\usepackage{enumerate}
\usepackage{mathtools}
\usepackage{url}
\usepackage{graphicx}
\usepackage{hyperref}

\usepackage{hyperref}\hypersetup{colorlinks}

\usepackage{color} 

\definecolor{darkred}{rgb}{1,0,0} 
\definecolor{darkgreen}{rgb}{0,0.8,0}
\definecolor{darkblue}{rgb}{0,0,1}

\hypersetup{colorlinks,
linkcolor=darkblue,
filecolor=darkgreen,
urlcolor=darkred,
citecolor=darkgreen}

 \numberwithin{equation}{section}

\newtheorem {Theorem}{Theorem}

 \numberwithin{Theorem}{section}

\newtheorem {Proposition}[Theorem]{Proposition}
\newtheorem {Corollary}[Theorem]{Corollary}
\theoremstyle{definition}
\newtheorem{Definition}[Theorem]{Definition}

\newtheorem{Remark}[Theorem]{Remark}
\newtheorem{Example}[Theorem]{Example}

\theoremstyle{remark}

\newtheoremstyle{TheoremForIntro} 
        {.6em}{.6em}              
        {\itshape}                      
        {}                              
        {\bfseries}                     
        {. }                             
        { }                             
        {\thmname{#1}\thmnote{ \bfseries #3}}
    \theoremstyle{TheoremForIntro}

\expandafter\chardef\csname pre amssym.def at\endcsname=\the\catcode`\@
\catcode`\@=11
\def\undefine#1{\let#1\undefined}
\def\newsymbol#1#2#3#4#5{\let\next@\relax
 \ifnum#2=\@ne\let\next@\msafam@\else
 \ifnum#2=\tw@\let\next@\msbfam@\fi\fi
 \mathchardef#1="#3\next@#4#5}
\def\mathhexbox@#1#2#3{\relax
 \ifmmode\mathpalette{}{\m@th\mathchar"#1#2#3}%
 \else\leavevmode\hbox{$\m@th\mathchar"#1#2#3$}\fi}
\def\hexnumber@#1{\ifcase#1 0\or 1\or 2\or 3\or 4\or 5\or 6\or 7\or 8\or
 9\or A\or B\or C\or D\or E\or F\fi}

\font\teneufm=eufm10
\font\seveneufm=eufm7
\font\fiveeufm=eufm5
\newfam\eufmfam
\textfont\eufmfam=\teneufm
\scriptfont\eufmfam=\seveneufm
\scriptscriptfont\eufmfam=\fiveeufm

\catcode`\@=\csname pre amssym.def at\endcsname

\newcommand{\fsl}{{\mathfrak {sl}}}
\newcommand{\fgl}{{\mathfrak {gl}}}
\newcommand{\fso}{{\mathfrak {so}}}

\newcommand{\fsp}{{\mathfrak {sp}}}

\newcommand{\fg}{{\mathfrak g}}
\newcommand{\fh}{{\mathfrak h}}

\newcommand{\fm}{{\mathfrak m}}

\newcommand{\fu}{{\mathfrak u}}

\newcommand{\sSU}{{\mathsf {SU}}}
\newcommand{\sSL}{{\mathsf {SL}}}
\newcommand{\sSO}{{\mathsf {SO}}}
\newcommand{\sO}{{\mathsf {O}}}
\newcommand{\sSp}{{\mathsf {Sp}}}

\newcommand{\sGL}{{\mathsf {GL}}}

\newcommand{\sG}{{\mathsf G}}

\newcommand{\sH}{{\mathsf H}}
\newcommand{\sS}{{\mathsf S}}

\newcommand{\sP}{{\mathsf P}}

\newcommand{\sU}{{\mathsf U}}

\newcommand{\Sym}{{\mathsf {Sym}}}

\newcommand{\Fuch}{{\mathsf{Fuch}}}
\newcommand{\Hom}{{\mathsf{Hom}}}

\newcommand{\Pic}{{\mathsf{Pic}}}
\newcommand{\Prym}{{\mathsf{Prym}}}
\newcommand{\End}{{\mathsf{End}}}
\newcommand{\Hit}{{\mathsf{Hit}}}
\newcommand{\Inn}{{\mathsf{Inn}}}

\newcommand{\Xx}{{\mathcal X}}

\newcommand{\Ff}{{\mathcal F}}

\newcommand{\Gg}{{\mathcal G}}

\newcommand{\Mm}{{\mathcal M}}

\newcommand{\Oo}{{\mathcal O}}
\newcommand{\Pp}{{\mathcal P}}

\newcommand{\mtrx}[1]{\left (\begin{matrix}#1\end{matrix}\right)}

\def    \I      {{\mathbb I}}
\def    \C      {{\mathbb C}}
\def    \R      {{\mathbb R}}
\def    \Z      {{\mathbb Z}}

\def    \ra     {{\rightarrow}}

\def    \haf    {{\frac{1}{2}}}

\def    \p      {\partial}

\def    \H  {\operatorname{\scriptscriptstyle{H}}}

\newcommand{\An}{\xymatrix{ *{\circ} \ar@{-}[r]|*\dir{ } & *{\circ}\ar@{-}[r]&{\cdots}&*{\circ}\ar@{-}[l]|*\dir{ }\ar@{-}[r]|*\dir{ }&*{\circ} }}
\newcommand{\Anlabel}{\xymatrix@R=.25em{ *{\circ}\ar@<-1ex>@{}[d]^{\alpha_{1}} \ar@{-}[r]|*\dir{ } & *{\circ}\ar@<-1ex>@{}[d]^{\alpha_{2}}\ar@{-}[r]&{\cdots}&*{\circ}\ar@{-}[l]|*\dir{ }\ar@{-}[r]|*\dir{ }\ar@<-2ex>@{}[d]^{\alpha_{n-1}}&*{\circ}\ar@<-1ex>@{}[d]^{\alpha_{n}}\\&&&& }}
\newcommand{\Bn}{\xymatrix{ *{\circ} \ar@{-}[r]|*\dir{ } & *{\circ}\ar@{-}[r]&{\cdots}&*{\circ}\ar@{-}[l]|*\dir{ }\ar@{=}[r]|*\dir{>}&*{\circ} }}
\newcommand{\Bnlabel}{\xymatrix@R=.25em{ *{\circ}\ar@<-1ex>@{}[d]^{\alpha_{1}} \ar@{-}[r]|*\dir{ } & *{\circ}\ar@<-1ex>@{}[d]^{\alpha_{2}}\ar@{-}[r]&{\cdots}&*{\circ}\ar@{-}[l]|*\dir{ }\ar@{=}[r]|*\dir{>}\ar@<-2ex>@{}[d]^{\alpha_{n-1}}&*{\circ}\ar@<-1ex>@{}[d]^{\alpha_{n}}\\&&&& }}
\newcommand{\Cn}{\xymatrix{ *{\circ} \ar@{-}[r]|*\dir{ } & *{\circ}\ar@{-}[r]&{\cdots}&*{\circ}\ar@{-}[l]|*\dir{ }\ar@{=}[r]|*\dir{<}&*{\circ} }}
\newcommand{\Cnlabel}{\xymatrix@R=.25em{ *{\circ}\ar@<-1ex>@{}[d]^{\alpha_{1}} \ar@{-}[r]|*\dir{ } & *{\circ}\ar@<-1ex>@{}[d]^{\alpha_{2}}\ar@{-}[r]&{\cdots}&*{\circ}\ar@{-}[l]|*\dir{ }\ar@{=}[r]|*\dir{<}\ar@<-2ex>@{}[d]^{\alpha_{n-1}}&*{\circ}\ar@<-1ex>@{}[d]^{\alpha_{n}}\\&&&& }}
\newcommand{\Dn}{\xymatrix@R=.25em{&&&&*{\circ} \\ *{\circ} \ar@{-}[r]|*\dir{ } & *{\circ}\ar@{-}[r]&{\cdots}&*{\circ}\ar@{-}[l]|*\dir{ }\ar@{-}[ur]|*\dir{ }\ar@{-}[dr]|*\dir{ }& \\ &&&&*{\circ} }}
\newcommand{\Dnlabel}{\xymatrix@R=.25em{&&&&*{\circ}\ar@<-1ex>@{}[d]^{\alpha_{n-1}} \\ *{\circ}\ar@<-1ex>@{}[d]^{\alpha_{1}} \ar@{-}[r]|*\dir{ } & *{\circ}\ar@<-1ex>@{}[d]^{\alpha_{2}}\ar@{-}[r]&{\cdots}&*{\circ}\ar@{-}[l]|*\dir{ }\ar@{-}[ur]|*\dir{ }\ar@{-}[dr]|*\dir{ }\ar@<-4ex>@{}[d]^{\alpha_{n-2}}& \\ &&&&*{\circ}\ar@<-2ex>@{}[u]^(-.5){\alpha_{n}} }}